\documentclass[12pt,reqno, aps]{amsproc}
\usepackage[latin1]{inputenc}
\usepackage[italian,english]{babel}
\usepackage{amssymb}
\usepackage{psfrag}
\usepackage{esint}
\usepackage{graphicx,epstopdf}
\usepackage{mathrsfs}
\usepackage{bbm}
\usepackage{paralist}
\newcommand{\Lim}[1]{\raisebox{0.6ex}{\scalebox{0.9}{$\displaystyle \lim_{#1}\;$}}}
\newcommand{\LimInf}[1]{\raisebox{0.6ex}{\scalebox{0.9}{$\displaystyle \liminf_{#1}\;$}}}
\newcommand{\Inf}[1]{\raisebox{0.6ex}{\scalebox{0.9}{$\displaystyle \inf_{#1}\;$}}}

\usepackage{color}
\definecolor{citation}{rgb}{0.2,0.58,0.2}
\definecolor{formula}{rgb}{0.1,0.2,0.6}
\definecolor{url}{rgb}{0.3,0,0.5}

\numberwithin{equation}{section}
\newtheorem{thm}{Theorem}[section]

\newtheorem{lem}[thm]{Lemma}
\newtheorem{cor}[thm]{Corollary}
\theoremstyle{definition}

\theoremstyle{remark}
\newtheorem{remark}[thm]{Remark}
\numberwithin{equation}{section}

\begin{document}

\title[Fractional Schr\"{o}dinger system of Choquard type]{\bf \large \textsc{ON FRACTIONAL
SCHR\"{O}DINGER SYSTEMS OF CHOQUARD TYPE}}
\author[SANTOSH BHATTARAI]{ SANTOSH BHATTARAI }
\address{SANTOSH BHATTARAI}
\email{sbmath55@gmail.com}


%\date{}

\maketitle

\begin{abstract}
\noindent In this article, we first employ the concentration compactness
techniques to prove existence and stability results of standing waves
for nonlinear fractional Schr\"{o}dinger-Choquard equation
\[
i\partial_t\Psi + (-\Delta)^{\alpha}\Psi = a |\Psi|^{s-2}\Psi+\lambda \left( \frac{1}{|x|^{N-\beta}} \star |\Psi|^p \right)|\Psi|^{p-2}\Psi\ \ \mathrm{in}\
\mathbb{R}^{N+1},
\]
where $N\geq 2$, $\alpha\in (0,1)$, $\beta\in (0, N)$, $s\in (2, 2+\frac{4\alpha}{N})$, $p\in [2, 1+\frac{2\alpha+\beta}{N})$, and the constants $a, \lambda$ are
nonnegative satisfying $a+\lambda > 0.$
We then extend the arguments to establish similar results for coupled standing waves
of nonlinear fractional Schr\"{o}dinger systems of Choquard type.
The same argument works for equations with an arbitrary number of combined nonlinearities
and when $|x|^{\beta-N}$ is replaced by a more general convolution potential $\mathcal{K}:\mathbb{R}^N\to [0, \infty)$ under certain assumptions.
The arguments can be applied and the results are identical for the case $\alpha=1$ as well.
\end{abstract}

%\tableofcontents
\section{Introduction}\label{SEC1}
The fractional nonlinear Schr\"{o}dinger (fNLS) equation
\begin{equation}\label{fNLS}
i\Psi_t+(-\Delta)^\alpha \Psi=f(|\Psi|)\Psi\ \ \mathrm{in}\ \mathbb{R}^N\times (0, \infty)
\end{equation}
is a fundamental equation of the space-fractional quantum mechanics (SFQM).
Here the fractional Laplacian $(-\Delta)^\alpha$ of order $\alpha\in(0, 1)$ is defined as
\begin{equation}\label{fracLap}
(-\Delta)^\alpha u(x)=C_{N, \alpha}~\mathrm{P.V.}\int_{\mathbb{R}^N}\frac{u(x)-u(y) }{|x-y|^{N+2\alpha} }dy,\ \forall u\in \mathcal{S}(\mathbb{R}^N),
\end{equation}
where P.V. stands for the principle value of the integral.
The term SFQM provides a natural extension of the standard quantum mechanics when
the Brownian trajectories in the well-known Feynman path integrals are
replaced by the L\'{e}vy  flights
(see \cite{Las3}).
The fNLS equations with $\alpha=1/2$ have been also used as models to describe Boson-stars. Recently, an optical realization of the fractional Schr\"{o}dinger equation was proposed by Longhi \cite{Lon1}.

\smallskip

An important issue concerning nonlinear evolution
equations such as \eqref{fNLS}
is to study
their standing wave solutions.
A standing wave solution of \eqref{fNLS} is a solution of the form
$
\Psi(x,t)=e^{-i\omega t}u(x),
$
where $\omega\in \mathbb{R}$ and $u$ satisfies the stationary equation
\begin{equation}\label{fODE}
(-\Delta)^\alpha u + \omega u= f(|u|)u\ \ \mathrm{in}\ \mathbb{R}^N.
\end{equation}
Equation \eqref{fODE} with the space derivative of order $\alpha=1$ (the standard Schr\"{o}dinger equation)
and its variants have been extensively studied in the mathematical literature.
Recently, there has been growing interest in extending similar results in the case $0<\alpha<1.$
One way to obtain solutions of \eqref{fODE} in $H^\alpha(\mathbb{R}^N)$ is to look for
the critical points of the functional $J_\omega:H^\alpha(\mathbb{R}^N)\to \mathbb{R}$ given by
\[
J_\omega(u)=\frac{1}{2}\|(-\Delta)^{\alpha/2}u\|_{L^2}^2 +\frac{\omega}{2}\|u\|_{L^2}^2 - \int_{\mathbb{R}^N}F(u)\ dx,
\]
where $F^\prime(u)=f(|u|)u.$ That is, ground state standing waves are characterized as solutions to the minimizing problem
\[
\inf_{ u \in \mathcal{N}} J_\omega(u),\ \ \mathcal{N}:=\{ u\in H^\alpha(\mathbb{R}^N)\setminus \{0 \}: J_\omega^{\prime}(u)(u)=0 \}.
\]
In this approach, since the parameter $\omega$ is assumed to
be fixed, we can not have a priori knowledge
about the $L^2$-norm of the standing wave profile.
A natural question concerning \eqref{fODE} is thus the case of normalized solutions, i.e., solutions $u$ satisfying
$
\|u\|_{L^2}^2=\sigma
$
for given $\sigma>0.$
This paper is devoted to the study of such special solutions
to fNLS equations and their coupled systems.

\smallskip

We first consider
\eqref{fODE} with combined power and Hartree type nonlinearities
\begin{equation}\label{Spower}
(-\Delta)^\alpha u + \omega u= a |u|^{s-2}u+\lambda \mathcal{I}_\beta^{p}(u)|u|^{p-2}u,
\end{equation}
where $a, \lambda$ are nonnegative satisfying $a+\lambda\neq 0$ and for any $r,$ the
Riesz potential $\mathcal{I}_{\beta}^{r}(u)$ of
order $\beta\in (0, N)$ for $u$ is defined by
\begin{equation}\label{IrDef}
\mathcal{I}_\beta^r(u)(x)=\frac{1}{|x|^{N-\beta}}\star |u|^r=\int_{\mathbb{R}^N}\frac{|u(y)|^r}{|x-y|^{N-\beta}}\ dy,\ x\in \mathbb{R}^N.
\end{equation}
A natural way to obtain normalized solutions
to \eqref{Spower} in $H^\alpha(\mathbb{R}^N)$ is to describe them
as solutions of the minimization problem constrained to the $L^2$ sphere of radius $\sigma,$
\begin{equation}\label{MinP1}
\mathrm{minimize}\ J(u)\ \mathrm{subject\ to}\ \int_{\mathbb{R}^N}|u|^2\ dx=\sigma>0,\ u\in H^\alpha(\mathbb{R}^N),
\end{equation}
where the functional $J$ represents the energy and is given by
\begin{equation}\label{EdefG}
J(u)=\frac{1}{2}\|(-\Delta)^{\alpha/2}u\|_{L^2}^2 - \frac{a}{s}\int_{\mathbb{R}^N}|u|^s\ dx-\frac{\lambda}{2p}\int_{\mathbb{R}^N}\mathcal{I}_\beta^{p}(u) |u|^p\ dx.
\end{equation}
Parameter $\omega$ in this approach becomes the Lagrange multiplier of the constrained variational problem.
The first part of this paper is devoted to proving
the existence of minimizers for the problem \eqref{MinP1}.

\smallskip

In the specific case $N=3, a=0$, $\beta=2\alpha=p=2,$ equation \eqref{Spower} reduces to
\begin{equation}\label{Cho2eq}
-\Delta u+\omega u=\lambda \left(\frac{1}{|x|}\star |u|^2 \right)u\ \ \mathrm{in}\ \mathbb{R}^3,
\end{equation}
and is also
known as
the Choquard-Pekar equation. Equation \eqref{Cho2eq} is the form that appears as a model in
quantum theory of a polaron at rest (see \cite{Pekar}).
The time-dependent form of \eqref{Cho2eq}
also describes the self-gravitational collapse of a quantum
mechanical wave-function (see \cite{Pen2}), in which context
it is usually called the Schr\"{o}dinger-Newton (SN) equation. In the plasma physics context,
the stationary form of the SN equation is also known as Choquard equation (see \cite{Lieb1} for details).

\smallskip

In pure mathematics, there is a huge literature concerning
ground states for the Choquard and related equations. Among many others,
we mention the paper by
Ma Li and Zhao Lin \cite{MaLi} where the existence problem of ground states for \eqref{Spower} with $\alpha=1$ and $a=0$ was
formulated, and the radial symmetry as well as the regularity of solutions have been
proved (see also \cite{Mor1}). Recently, the
ground state solution of \eqref{Spower} with $a=0$ has been studied by
D'Avenia et al. in
\cite{Dav1}.
Surprisingly, only a very few papers address normalized solutions, although they are
the most relevant from the physical point of view.
In \cite{Lieb1}, Lieb proved
the existence and uniqueness of normalized solutions
to the Choquard equation using the symmetrization techniques. In their papers \cite{CLi, Lio}, Cazenave and Lions studied
the existence and stability issues of normalized solutions for the Choquard and related equations.

\smallskip

Another interesting question is whether similar results can be proved for coupled systems of
fNLS equations. In the past,
systems of standard NLS equations (in the case $\alpha=1$) have been
widely studied and a fairly
complete theory has been developed to study
standing wave solutions to such systems
(see \cite{ABnls}, for instance, concerning the results on normalized solutions).
No such theory yet exists, however, for the coupled systems of fNLS equations.
In this paper, we are also interested in generalizing existence and stability results to
the following coupled system of fNLS equations with Choquard type nonlinearities
\begin{equation}\label{CfODEp}
\left\{
\begin{aligned}
  & (-\Delta)^\alpha u_1+\omega_1u_1=\lambda_1\mathcal{I}_\beta^{p_1}(u_1) |u_1|^{p_1-2}u_1 +c \mathcal{I}_\beta^{q}(u_2)  |u_1|^{q-2}u_1,\\
   & (-\Delta)^\alpha u_2+\omega_2u_2= \lambda_2 \mathcal{I}_\beta^{p_2}(u_2)|u_2|^{p_2-2}u_2+c \mathcal{I}_\beta^{q}(u_1) |u_2|^{q-2}u_2,
\end{aligned}
\right.
\end{equation}
where $u_1, u_2:\mathbb{R}^N\to \mathbb{C}$, $\omega=(\omega_1, \omega_2)\in \mathbb{R}^2,$ the
constants $\lambda_1, \lambda_2,$ and $ c$ are
positive, and $\mathcal{I}_\beta^r(f)$ is as defined in \eqref{IrDef}.
As in the scalar case, we search for normalized solutions to \eqref{CfODEp}, i.e.,
solutions $(u_1, u_2)$ in $H^\alpha(\mathbb{R}^N)\times H^\alpha(\mathbb{R}^N)$ satisfying
$\|u_1\|_{L^2}^2=\sigma_1$ and $\|u_2\|_{L^2}^2=\sigma_2$ for given $\sigma_1>0$ and $\sigma_2>0.$
To obtain normalized solutions to \eqref{CfODEp}, we look for minimizers of the problem
\begin{equation}\label{MinP12}
\mathrm{minimize}\ E(u_1, u_2)\ \mathrm{subject\ to}\ \int_{\mathbb{R}^N}|u_j|^2\ dx=\sigma_j>0,\ j=1,2,
\end{equation}
where the associated energy functional $E$ is given by
\begin{equation}\label{EdefGC}
\begin{aligned}
E(u_1, u_2)= \frac{1}{2}\sum_{j=1}^2 \|(-\Delta)^{\alpha/2}u_j\|_{L^2}^2& -\sum_{j=1}^2 \frac{\lambda_j}{2p_j}\iint_{\mathbb{R}^N}\frac{|u_j(x)|^{p_j}|u_j(y)|^{p_j} }{|x-y|^{N-\beta} }\ dxdy\\
& - \frac{c}{q} \iint_{\mathbb{R}^N}\frac{|u_1(x)|^q|u_2(y)|^q}{|x-y|^{N-\beta}}\ dxdy.
\end{aligned}
\end{equation}
Parameters $\omega_1$ and $\omega_2$ in \eqref{CfODEp} appear as the Lagrange multipliers.
Apart from the existence result for minimizers of \eqref{MinP12}, we also
provide some results concerning structures of the set of minimizers.

\smallskip

All results established in Section~\ref{singleEq} below are easily extendable to versions of \eqref{Spower} with an
arbitrary number of combined nonlinearities and when $|x|^{\beta-N}$ is replaced by a more general
convolution potential $\mathcal{K}:\mathbb{R}^N\to [0, \infty)$, that is, the equation of the form
\begin{equation}\label{arbiNon}
(-\Delta)^\alpha u+\omega u = \sum_{j=1}^m a_j |u|^{s_j-2}u+\sum_{k=1}^d \lambda_k (\mathcal{K}\star |u|^{p_k})|u|^{p_k-2}u,
\end{equation}
where $N\geq 3,\ 0<\alpha<1$, the constants $a_j, \lambda_k$ are all
nonnegative but not all zero, and the potential $\mathcal{K}\in L_{\mathrm{w} }^r(\mathbb{R}^N)$ is radially
symmetric satisfying some assumptions (see Theorem~\ref{arbiTHM} below). Normalized
solutions to \eqref{arbiNon} in $H^\alpha(\mathbb{R}^N)$ are obtained as
minimizers of
the energy functional
\begin{equation}\label{tildDef}
\tilde{J}(u)=\frac{1}{2}\|(-\Delta)^{\alpha/2}u\|_{L^2}^2-\sum_{j=1}^mb_j\|u\|_{L^{s_j}}^{s_j}
-\sum_{k=1}^d\lambda_k\int_{\mathbb{R}^N}(\mathcal{K}\star |u|^{p_k})|u|^{p_k}\ dx
\end{equation}
satisfying the constraint $\|u\|_{L^2}^2=\sigma>0,$
where $b_j=a_j/s_j$ and $\mu_k=\lambda_k/2p_k.$

\smallskip

Our approach here involves in studying the global
minimization properties of stationary solutions via
the concentration-compactness principle
of P. L. Lions (see Lemma I.1 of \cite{Lio}).
The advantage of utilizing this technique in our context
is that this not only gives the existence of stationary solutions but
also addresses the important
stability issue of associated standing waves,
since the energy and the power(s) involved in variational
problems are conservation laws for
the flow of associated NLS-type evolution equations (see \cite{CLi,[A]} for the illustration of the
method).
To study the two-parameter problem \eqref{MinP12}, we follow the
techniques developed in a series of papers \cite{AAkdv, ABnls, [AB11]} where the concentration compactness principle
was used to study solitary waves for coupled Schr\"{o}dinger and KdV systems.
In order to establish relative compactness of energy minimizing sequences
(and hence, existence and stability of minimizers)
in the spirit of
concentration compactness technique, one require to check certain strict
inequalities involving the infimum of the minimization problem.
The proof of the strict inequality is a notable difficulty when one uses this technique and
most proofs of these strict inequalities in the
literature are based on some homogeneity-type property.
The proof is much less understood
for problems which violate homogeneity-type assumptions and for multi-constrained
variational problems. This might be one reason why there are only a very few papers concerning
the relative compactness of minimizing sequences, and hence existence and stability of normalized standing waves.

\smallskip

The paper is organized as follows. We analyze separately the scalar case and the coupled
one. In Section~\ref{singleEq}, we analyze the fNLS equation with combined local and
nonlocal nonlinearities via concentration compactness method.
Theorem~\ref{Mainpower}, Corollary~\ref{Maincor}, and Theorem~\ref{arbiTHM} are the main results of Section~\ref{singleEq}.
Section~\ref{coupledS} provides the analysis of the two-constraint problem for the coupled fNLS system with nonlocal nonlinearities.
Theorem~\ref{CHMain} and its Corollary are the existence and stability results for coupled standing waves.
To the author's knowledge, this is the first paper
which proves existence results of normalized solutions
to coupled Schr\"{o}dinger systems involving convolution type
interaction terms. As far as we know, there is also no paper which
employs
the concentration compactness technique (in the presence of strict inequality) to obtain
standing waves for nonlinear Schr\"{o}dinger equations with an arbitrary
number of combined nonlinearities.

\medskip

\textbf{Notation.} The ball with radius $R$ and center $x\in \mathbb{R}^N$ will be denoted by $B_R(x).$
For any $r\geq 1,$ we denote by $L^r(\mathbb{R}^N)$ the space of all complex-valued $r$-integrable functions $f$ with the norm
$\|f\|_{L^r}^r=\int_{\mathbb{R}^N}|f|^r\ dx.$ For any $1\leq r<\infty,$ the \textit{weak}-$L^r$ space $L_{\mathrm{w}}^r(\mathbb{R}^N)$ is the set of
all measurable functions $f:\mathbb{R}^N\to \mathbb{C}$ such that
\[
\|f\|_{L_{\mathrm{w}}^r}:=\sup_{M>0}M \left\vert \{x:|f(x)|>M \} \right\vert^{1/r}<\infty.
\]
We have the proper
inclusion $L^r(\mathbb{R}^N)\hookrightarrow L_{\mathrm{w}}^r(\mathbb{R}^N)$ (see \cite{Lieb} for details).
The fractional Laplacian $(-\Delta)^ \alpha, \alpha\in(0,1),$ defined in \eqref{fracLap}, can
be equivalently defined via Fourier transform as
\[
\widehat{(-\Delta)^\alpha u}(\xi)=|\xi|^{2 \alpha}\widehat{u}(\xi),\  \forall u\in \mathcal{S}(\mathbb{R}^N).
\]
We denote by $H^\alpha(\mathbb{R}^N), 0<\alpha<1,$ the fractional Sobolev space of all complex-valued functions
functions $u\in L^2(\mathbb{R}^N)$ with the norm
\[
\|u\|_{H^\alpha}^2=\|(-\Delta)^{\alpha/2}u\|_{L^2}^2+\|u\|_{L^2}^2,
\]
where up to a multiplicative constant
\[
\|(-\Delta)^{\alpha/2}u\|_{L^2}^2=\iint_{\mathbb{R}^N}\frac{|u(x)-u(y)|^2 } {|x- y|^{N+2\alpha} }dxdy,
\]
also known as Gagliardo seminorm. Throughout the paper, we denote by $Y^\alpha(\mathbb{R}^N)$ the product space
$Y^\alpha(\mathbb{R}^N)=H^\alpha(\mathbb{R}^N)\times H^\alpha(\mathbb{R}^N)$ and its norm by
\[
\|(u,v)\|_{Y^\alpha}^2=\|u\|_{H^\alpha}^2+\|v\|_{H^\alpha}^2.
\]
We do not develop
the elements of the theory of fractional Sobolev spaces here, but use a
number of fractional Sobolev type inequalities throughout the paper (for a detailed
account of fractional Sobolev spaces, the reader may consult \cite{NPV}). The same symbol $C$ will frequently be
used to denote different constants in the same
string of inequalities whose exact value is not important for our analysis.

\section{Fractional Schr\"{o}dinger-Choquard equation }\label{singleEq}
In this section we study existence and orbital stability issues of standing waves for fNLS equation with combined power and Choquard type nonlinearities
\begin{equation}\label{NSpower}
i\Psi_t+(-\Delta)^\alpha \Psi = a |\Psi|^{s-2}\Psi+\lambda \mathcal{I}_\beta^{p}(\Psi)|\Psi|^{p-2}\Psi,\ (x,t)\in \mathbb{R}^N\times (0, \infty),
\end{equation}
where $N\geq 2$ and $a, \lambda$ are nonnegative constants satisfying $a+\lambda \neq 0.$
Throughout this section, we assume
that the
following conditions hold:
\begin{equation}\label{AssPo}
N\geq 2, \beta\in (0,N), 0<\alpha<1,\ 2<s<2+\frac{4\alpha}{N},\ 2\leq p< \dfrac{N+2\alpha+\beta}{N}.
\end{equation}

\smallskip

We first provide the statement of our main results.
For any $\sigma>0,$
we denote by $\Sigma_\sigma$ the set of fixed power
$\Sigma_\sigma=\{u \in H^\alpha(\mathbb{R}^N):\|u \|_{L^2}^2=\sigma\}$ and define
\[
J(\sigma)=\inf_{f\in \Sigma_\sigma}J(f).
\]
The following theorem and its corollary are the main results of this section:
\begin{thm}\label{Mainpower}
Suppose that \eqref{AssPo} holds.
For every $\sigma>0,$ let $P_\sigma$ denotes the set of standing wave profiles of \eqref{NSpower}, that is,
\[
P_\sigma=\left\{u\in H^\alpha(\mathbb{R}^N):u\in \Sigma_\sigma\ \mathrm{and}\ J(u)=J(\sigma) \right\}.
\]
Then, the set
$P_\sigma$ is nonempty. If $u\in P_\sigma,$
then $|u|\in P_\sigma$ and $|u|>0$ on $\mathbb{R}^N.$ Moreover, $|u|^\ast\in P_\sigma$ whenever $u\in P_\sigma$, where $f^\ast$ represents
the symmetric decreasing rearrangement of $f.$

\smallskip

Furthermore, if $f_n\in H^\alpha(\mathbb{R}^N )$,
$\|f_n\|_{L^2}^2\to \sigma,$ and $J(f_n)\to J(\sigma),$ then
\begin{itemize}
\item[$1.$] the sequence $\{f_n\}_{n\geq 1}$ is relatively compact in $H^\alpha(\mathbb{R}^N)$ up to a translation, i.e, there exists a subsequence which we still denote by the same, a sequence $\{y_1, y_2, \ldots\}\subset \mathbb{R}^N,$ and a function $u\in H^\alpha(\mathbb{R}^N)$ such that $f_n(\cdot+y_n)\to u$ strongly in $H^\alpha(\mathbb{R}^N)$ and $u\in P_\sigma.$
\item[$2.$] The following holds:
\[
\lim_{n\to \infty}\inf_{u\in P_\sigma}\inf_{y\in \mathbb{R}^N}\|f_n(\cdot+y)-u\|_{H^\alpha} =0.
\]
\item[$3.$] $f_n\to P_\sigma$ in the following sense,
\[
\lim_{n\to \infty}\inf_{u\in P_\sigma}\|f_n-u\|_{H^\alpha}=0.
\]
\end{itemize}
\end{thm}
 A straightforward consequence is the following result.
\begin{cor}\label{Maincor}
For every $\sigma>0,$
the solution set $P_\sigma$ is stable in the following sense: for any $\varepsilon>0,$ there exists $\delta>0$ such that if $u_0\in H^\alpha(\mathbb{R}^N),$
$u\in P_\sigma,$ and $\|u_0-u\|_{H^\alpha}<\delta,$ then the solution $\Psi(x,t)$ of \eqref{NSpower} with $\Psi(x,0)=u_0(x)$ satisfies
\[
\inf_{v\in P_\sigma}\|\Psi(\cdot, t)-v(x)\|_{H^\alpha}<\varepsilon,\ \forall t\geq 0.
\]
\end{cor}
In order to obtain analogue results for the equation \eqref{arbiNon} involving an arbitrary
number of combined power and Hartree type nonlinearities, we require that
the powers $s_j$ and $p_k$ satisfy
\begin{equation}\label{assCom}
2<s_j<2+\frac{4 \alpha}{N},\ 1 \leq j \leq m,\  \mathrm{and}\  2\leq p_k<\frac{2r(\alpha+N)-N}{Nr},\ 1\leq k\leq d.
\end{equation}
Furthermore, we require that the convolution potential $\mathcal{K}(x)$ satisfies the following assumptions
\begin{itemize}
\item[(H1)] $\mathcal{K}:\mathbb{R}^N\to [0, \infty)$ is radially symmetric, i.e., $\mathcal{K}(x)=\mathcal{K}(|x|),$ satisfying
\[
\lim_{r\to \infty}\mathcal{K}(r)=0\ \ \mathrm{and}\ \mathcal{K}\in L_{\mathrm{w}}^r(\mathbb{R}^N)\ \ \mathrm{for}\ r>\frac{N}{2N+2\alpha-Np_k},\ \forall 1\leq k\leq d.
\]
\item[(H2)] $\mathcal{K}$ satisfies the following condition
\[
\forall \theta>1,\ \mathcal{K}\left(\theta \xi\right)\geq \left( \frac{1}{\theta}\right)^\Gamma \mathcal{K}(\xi)\ \  \mathrm{for\ some}\ \Gamma>0.
\]
\item[(H3)] The number $\Gamma$ satisfies $\Gamma<2\alpha+N(2-p_k)$ for all $1\leq k\leq d.$

\end{itemize}

Analogue existence result for \eqref{arbiNon} is the following:
\begin{thm}\label{arbiTHM}
Suppose that $\mathcal{K}$ satisfy the assumptions (H1), (H2), (H3), and the powers $s_j$ and $p_k$ satisfy the conditions \eqref{assCom}.
For any $\sigma>0,$ define
\begin{equation}\label{Var3}
\tilde{J}(\sigma)=\Inf{f\in \Sigma_{\sigma}}\tilde{J}(f)\ \ \mathrm{and}\ \widetilde{P}_\sigma=\left\{u\in H^\alpha(\mathbb{R}^N): u\in \Sigma_\sigma, \tilde{J}(u)=\tilde{J}(\sigma) \right\},
\end{equation}
where $\Sigma_\sigma$ and $\tilde{J}(f)$ are as defined in Section~\ref{SEC1}.
Then, the solution set
$\widetilde{P}_\sigma$ for the problem $\tilde{J}(\sigma)$ is nonempty and the compactness
property is also enjoyed by its minimizing sequences. Each function $u\in \widetilde{P}_\sigma$ solves
the equation \eqref{arbiNon} for some $\omega>0.$
Furthermore, if $u\in \widetilde{P}_\sigma$, then $|u|\in \widetilde{P}_\sigma$ and $|u|>0.$
\end{thm}

The proof of Theorem~\ref{arbiTHM} is quite similar to the proof of Theorem~\ref{Mainpower}
and we do not provide all details of it. We only indicate the parts which require changing.

\smallskip

For the reader's convenience, we first recall some well-known results from the fractional Sobolev spaces.
The following lemma is the fractional
Sobolev embedding.
\begin{lem}\label{sobInq}
Let $0<\alpha<1$ be such that $N>2\alpha.$ Then there exists a constant $C=C(N,\alpha)>0$ such that
\begin{equation}\label{SobInq}
\|u\|_{L^{2N/(N-2\alpha)}}^2\leq C \|(-\Delta)^{\alpha/2}u\|_{L^2}^2,\  \forall u\in H^\alpha(\mathbb{R}^N),
\end{equation}
Thus, $H^\alpha (\mathbb{R}^N)$ is continuously embedded into
$L^r(\mathbb{R}^N)$
for any $2\leq r\leq \frac{2N}{N-2\alpha}$ and compactly embedded into $L_{\textrm{loc}}^r(\mathbb{R}^N)$
for every $2\leq r<\frac{2N}{N-2\alpha}.$
\end{lem}
\begin{proof}
See for example, Theorem~6.5 of \cite{NPV}.
\end{proof}

We require the fractional Gagliardo-Nirenberg inequality.
\begin{lem}
Let $1\leq r<\infty, 0<\alpha<1,$ and $N>2\alpha.$ Then, for any $u\in H^\alpha(\mathbb{R}^N),$ one has
\[
\|u\|_{L^r}\leq C \|(-\Delta)^{\alpha/2}u\|_{L^2}^\rho \|u\|_{L^t}^{1-\rho},
\]
where $t\geq 1, \rho\in [0,1],$ $C=C({N,\alpha, \rho})$, and $\rho$ satisfies the identity
\[
\frac{N}{r}=\frac{\rho(N-2\alpha)}{2}+\frac{N(1-\rho)}{t}.
\]
\end{lem}
\begin{proof}
It is a consequence of the
the H\"{o}lder and fractional Sobolev-type inequalities. If $r=1,$ it is obvious. For $r>1,$ the H\"{o}lder inequality yields
\[
\|u\|_{L^r}\leq \|u\|_{L^{2N/(N-2\alpha) } }^{\rho}\|u\|_{L^t }^{1-\rho},\ \textrm{where}\ \ \frac{1}{r}=\frac{1-\rho}{t}+\frac{\rho(N-2\alpha)}{2N}.
\]
Using the Sobolev inequality \eqref{SobInq}, we obtain that
\[
\|u\|_{L^r}\leq (C({N,\alpha}))^{\rho/2} \|(-\Delta)^{\alpha/2}u\|_{L^2}^\rho \|u\|_{L^t }^{1-\rho},
\]
which gives the desired inequality.
\end{proof}
We will make frequent use of the Hardy-Littlewood-Sobolev inequality:
\begin{lem}
Suppose $1<r, t<+\infty$ and $0<\gamma<N$ with $\frac{1}{r}+\frac{1}{t}+\frac{\gamma}{N}=2.$ If $u\in L^r(\mathbb{R}^N)$ and
$v\in L^t(\mathbb{R}^N),$ then there exists $C(r,t, \alpha, N)>0$ such that
\[
\iint_{\mathbb{R}^N}|u(x)||x-y|^{-\gamma}|v(y)|\ dxdy  \leq C(r,t,\alpha, N)\|u\|_{L^r}\|v\|_{L^t}.
\]
\end{lem}
\begin{proof}
See Theorem~4.3 of \cite{Lieb}.
\end{proof}
We also need the weak version of Young's inequality for convolutions which states that for any three measurable functions $f\in L^q(\mathbb{R}^N), g\in L_{ \mathrm{w}}^r(\mathbb{R}^N),$ and $h\in L^t(\mathbb{R}^N),$
\begin{equation}\label{weakY}
\left\vert \iint_{\mathbb{R}^N}f(x)g(x-y)h(x)\ dxdy \right\vert\leq C(q, N, r)\|f\|_{L^q}\|g\|_{L_{\mathrm{w}}^r}\|h\|_{L^t},
\end{equation}
where $1<q, r, t<\infty$ and
\[
\frac{1}{q}+\frac{1}{r}+\frac{1}{t}=2.
\]
The weak Young's inequality \eqref{weakY} was proved by Lieb \cite{Lieb3} as a corollary of the Hardy-Littlewood-Sobolev inequality.

\begin{remark}
In view of the weak version of Young's inequality, we see that the integral
\[
\iint_{\mathbb{R}^N}\mathcal{K}(x-y)|u(x)|^p|u(y)|^p\ dxdy
\]
is well-defined if $|u|^p\in L^t(\mathbb{R}^N)$ for all $t>1$ satisfying the condition
\[
\frac{2}{t}+\frac{1}{r}=2,\ \ \mathrm{or}\ \ t=\frac{2r}{2r-1}.
\]
In the present context, since $u\in H^\alpha(\mathbb{R}^N),$ we must require $tp\in [2, \frac{2N}{N-2\alpha}].$
By our assumption on $p=p_k,$ it follows that
\begin{equation}\label{checkAss}
\frac{1}{tp}=\frac{2r-1}{2pr}=\frac{1}{p}-\frac{1}{2pr}>\frac{1}{p}-\frac{2N+2\alpha-pN}{2pN}=\frac{1}{2}-\frac{\alpha}{pN}>\frac{1}{2}-\frac{\alpha}{N},
\end{equation}
and so, we have that $|u|^{p}\in L^{\frac{2r}{2r-1} }(\mathbb{R}^N)$ for every $u\in H^\alpha(\mathbb{R}^N).$ When
the convolution potential is $\mathcal{K}(x)=\mathcal{K}_\beta(x)=|x|^{-(N-\beta)},$ one has
$\mathcal{K}_\beta\in L_{\mathrm{w}}^r(\mathbb{R}^N)$ if and only if $N-\beta=N/r$ and \eqref{checkAss} reduces to
\[
\frac{N+\beta}{2Np}=\frac{1}{2p}+\frac{\beta}{2Np}>\frac{1}{2p}+\frac{pN-N-2\alpha}{2Np}
=\frac{1}{2}-\frac{\alpha}{Np}>\frac{N-2\alpha}{2N},
\]
and so, $|u|^p\in L^{\frac{2N}{N+\beta}}(\mathbb{R}^N)$ for
every $u\in H^\alpha(\mathbb{R}^N).$
One can also see that the condition (H3), namely $\Gamma<2\alpha +N(2-p),$ is equivalent to $p<(N+2\alpha+\beta)/N.$
These observations illustrate that the assumptions on the powers $p$ and $p_k$ as given in \eqref{AssPo} and \eqref{assCom} are
quite natural in the present setting.
\end{remark}

We now establish some important properties of the function $J(\sigma).$
We have broken the proof into several lemmas so that later, in the case of coupled system of fNLS equations, it will be
easy to identify the parts which require changing.

\smallskip

In what follows, we denote by $b=a/s$ and $\mu=\lambda/2p$ (see definition \eqref{EdefG} of the energy functional).
We denote for any $r>0,$
\[
D_r(|f|,|g|)=\frac{|f(x)|^r |g(y)|^r } {|x-y|^{N-\beta} },\ x,y\in \mathbb{R}^N.
\]
With this notation, the Coulomb energy functional has the following form
\[
 \mathbb{D}_{r}(f, g)=\int_{\mathbb{R}^N}\mathcal{I}_\beta^r(g)f(x)\ dx=\iint_{\mathbb{R}^N}D_r(|f|,|g|)\ dxdy.
\]
In particular, we simply write $\mathbb{D}_{r}(f)$ for the functional $\mathbb{D}_{r}(f, f).$

\begin{lem}\label{NegHP}
Suppose that \eqref{AssPo} holds. For any $\sigma>0,$
the following statements hold.

\smallskip

(i)\ If $f_n\in H^\alpha(\mathbb{R}^N)$, $\|f_n\|_{L^2}^2\to \sigma,$
and $J(f_n)\to J(\sigma),$ then there exists $B>0$ such that $\|f_n\|_{H^\alpha}\leq B$ for all $n.$

\smallskip

(ii) $J(\sigma)$ is bounded from below on $\Sigma_\sigma.$ Moreover, $J(\sigma)<0.$

\end{lem}
\begin{proof}
Applying the fractional Gagliardo-Nirenberg inequality and
using the boundedness of $\|f_n\|_{L^2},$ we obtain
\begin{equation}\label{Estp17}
\begin{aligned}
\|f_n\|_{L^{s}}^{s}\leq C \|(-\Delta)^{\alpha/2}f_n\|_{L^2}^{\lambda^\prime} \|f_n\|_{L^2}^{1-{\lambda^\prime}}  \leq C \|(-\Delta)^{\alpha/2}f_n\|_{L^2}^{\lambda^\prime}
 \leq C \|f_n\|_{H^\alpha}^{\lambda^\prime},
\end{aligned}
\end{equation}
where $\lambda^\prime=N(s-2)/2\alpha s.$
Next, by our assumption on $p$, we have
$|f|^p\in L^{\frac{2N}{N+\beta}}(\mathbb{R}^N)$ for
every $f\in H^\alpha(\mathbb{R}^N).$ Applying
the Hardy-Littlewood-Sobolev inequality with $r=t=2N/(N+\beta)$ and the fractional Gagliardo-Nirenberg inequality, we get
\begin{equation}\label{Mbou3}
\begin{aligned}
\iint_{\mathbb{R}^N}D_p(|f_n|,|f_n|)\ dx dy  & \leq C \| |f_n|^{p}\|_{L^{2N/(N+\beta)}}^2 =
C \|f_n\|_{L^{2Np/(N+\beta)}}^{2p} \\
& \leq C  \|(-\Delta)^{\alpha/2}f_n\|_{L^2}^{2p\mu^\prime}\|f_n\|_{L^2}^{2p(1-\mu^\prime)}\\
& \leq C \|(-\Delta)^{\alpha/2}f_n\|_{L^2}^{2p\mu^\prime}\leq C \|f_n\|_{H^\alpha}^{2p\mu^\prime},
\end{aligned}
\end{equation}
where $\mu^\prime=(Np-N-\beta)/2\alpha p.$
We now write
\[
\frac{1}{2}\|f_n\|_{H^\alpha}^2= J(f_n)+b \|f_n\|_{L^s}^s+\mu \iint_{\mathbb{R}^N}D_p(|f_n|,|f_n|)\ dx dy  + \frac{1}{2}\|f_n\|_{L^2}^2.
\]
Since $\{J(f_n)\}_{n\geq 1}$ is a convergent sequence of numbers, so it is bounded.
Utilizing the estimates \eqref{Estp17} and \eqref{Mbou3}, the above identity implies that
\begin{equation}\label{addS1}
\frac{1}{2}\|f_n\|_{H^\alpha}^2 \leq C \left( 1+ \|f_n\|_{H^\alpha}^{\lambda^\prime}+ \|f_n\|_{H^\alpha}^{2p\mu^\prime} \right).
\end{equation}
By our assumption on $p,$ we have that $pN-N-\beta<N$ and so
\[
2p\mu^\prime=\frac{pN-N-\beta}{\alpha}<\frac{N}{\alpha}<\frac{2\alpha}{\alpha}=2.
\]
Since $\lambda^\prime\in (0, 2)$ and $0<\frac{N-\beta}{\alpha}\leq 2p\mu^\prime <2,$ it follows from \eqref{addS1} that
that there exists a constant $B>0$ such that $\|f_n\|_{H^\alpha}\leq B$ for each $n.$
The statement that $J(\sigma)>-\infty$ easily follows from the estimates \eqref{Estp17} and \eqref{Mbou3}.

\smallskip

To show $J(\sigma)<0,$ first we observe that
if $u_\theta(\cdot)=\theta^A u(\theta^B \cdot),$ for $A, B\in \mathbb{R}$ and $\theta>0,$ then
we have that for any $p,$
\[
\mathcal{I}_\beta^p(u_\theta)(x)
 = \theta^{pA+B(N-\beta)}\int_{\mathbb{R}^N} \frac{ | u(\theta^B y)|^p} {|\theta^B x - \theta^B y|^{N-\beta} }dy
 =\theta^{pA-B\beta}\mathcal{I}_\beta^p(u)(\theta^B x).
\]
Now let $f\in \Sigma_\sigma$ be fixed.
For any $\theta>0,$ define the scaled function $f_\theta(x)=\theta^{1/2}f(\theta^{1/N}x).$ Then $f_\theta\in \Sigma_\sigma$ as well. If both $a>0$ and $\lambda>0;$ or $a=0$ and $\lambda>0,$ then we have that
\begin{equation}\label{NegAIn}
J(f_\theta)\leq \frac{\theta^{2\alpha/N}}{2}\|(-\Delta)^{\alpha/2}f\|_{L^2}^2-\mu \theta^{\kappa }\iint_{\mathbb{R}^N}D_p(|f|,|f|)\ dx dy,
\end{equation}
where $\kappa=(Np-\beta-N)/N.$
Since $\kappa$ satisfies
$0<\frac{N-\beta }{N}\leq \kappa<\frac{2\alpha}{N},$ we obtain that $ J(f_\theta)<0$ for
sufficiently small $\theta.$
This proves that $J(\sigma)\leq J(f_\theta)<0.$ When $\lambda=0,$ we have that
\[
J(f_\theta)\leq \frac{\theta^{2\alpha/N}}{2}\|(-\Delta)^{\alpha/2}f\|_{L^2}^2- b\theta^{(s-2)/2}\|f\|_{L^s}^s.
\]
Since the condition $2<s<2+\frac{4\alpha}{N}$ implies $(s-2)/2\in (0, 2\alpha/N),$ we again obtain that $ J(f_\theta)<0$ for
sufficiently small $\theta.$
\end{proof}

\begin{remark}
Analogue of Lemma~\ref{NegHP} for
$\widetilde{J}(\sigma)$ can be
proved by
applying the weak version of Young's inequality and the fractional Gagliardo-Nirenberg inequality.
To see this, observe first that our assumption on $p_k$ guarantees
$|f|^{p_k}\in L^{\frac{2r}{2r-1} }$ for every $f\in H^\alpha(\mathbb{R}^N)$ and all $1\leq k \leq d.$
Then, analogue of estimate \eqref{Mbou3} for any
minimizing sequence $\{f_n\}_{n\geq 1}$ of $\widetilde{J}(\sigma)$
takes the form
\[
\begin{aligned}
\int_{\mathbb{R}^N}(\mathcal{K}\star |f_n|^{p_k})|f_n|^{p_k}\ dx & \leq
 \|\mathcal{K}\|_{L_{\mathrm{w} }^r}\|f_n\|_{L^{\frac{2p_kr} {2r-1}} }^{2p_k}
 \leq C \|f_n\|_{L^2}^{2(1-\mu_k^\prime)p_k}\|\nabla f_n\|_{L^2}^{2\mu_k^\prime p_k}\\
& \leq C\|\nabla f_n\|_{L^2}^{2\mu_k^\prime p_k},\ \forall\ 1\leq k\leq d,
\end{aligned}
\]
where the numbers $\mu_k^\prime$ are given by
\[
\mu_k^\prime=\frac{N(rp_k-2r+1)}{2r\alpha p_k},\ 1\leq k\leq d,
\]
and by our assumption on $p_k$, the powers $2\mu_k^{\prime} p_k$ satisfy
\[
\begin{aligned}
 2\mu_k^{\prime} p_k=\frac{Np_k}{\alpha}-\frac{2N}{\alpha}+\frac{N}{r\alpha}<\frac{N}{\alpha}\left(\frac{2\alpha}{N}+2-\frac{1}{r} \right)-\frac{2N}{\alpha}+\frac{N}{r\alpha}=2.
\end{aligned}
\]
Analogue of the estimate \eqref{Estp17} remains true for each $s=s_j, 1\leq j\leq m.$
Then, an analogue of estimate \eqref{addS1} proves that any minimizing sequence $\{f_n\}_{n\geq 1}$ of
$\widetilde{J}(\sigma)$ is bounded in $H^\alpha(\mathbb{R}^N).$
The proof that $\widetilde{J}(\sigma)\in (-\infty, 0)$ will go through unchanged.
\end{remark}

\begin{lem}\label{PoLem1}
Suppose that $\lambda\neq 0.$
If $f_n\in H^\alpha(\mathbb{R}^N), \|f_n\|_{L^2}^2\to \sigma$, and $J(f_n)\to J(\sigma),$
then
there exists $\delta>0$ and $n_0\in \mathbb{N},$ depending on $\delta,$ such that for every $n\geq n_0,$
\[
\mathbb{D}_{p}(f_n)=\iint_{\mathbb{R}^N}D_p(|f_n|,|f_n|)\ dx dy =\int_{\mathbb{R}^N}\mathcal{I}_\beta^{p}(f_n)|f_n|^{p}\ dx \geq \delta.
\]
Moreover, for any $T>1,$ there exists $n_0\in \mathbb{N},$ depending on $T,$ such that for all $n\geq n_0,$
\[
J(T^{1/2} f_n)|_{a=0}
< T J(f_n)|_{a=0}.
\]
If $a> 0$ and $\lambda\geq 0,$ the same conclusions hold with $\mathbb{D}_{p}(f_n)$ replaced by $\|f_n\|_{L^s}^s$ and $J(g)|_{a=0}$ replaced by $J(g)|_{\lambda=0}.$
\end{lem}
\begin{proof}
To prove the first statement, suppose to the contrary that
$
\LimInf{n\to \infty}\mathbb{D}_{p}(f_n) = 0.
$
Then, it is obvious that
\begin{equation}\label{NegNlem}
J(\sigma)=\lim_{n\to \infty}J\left(f_n\right)\geq \liminf_{n\to \infty}\|(-\Delta)^{\alpha/2}f_n\|_{L^2}^2\geq 0,
\end{equation}
which contradicts the fact $J(\sigma)<0$ and hence, the first statement follows.
To prove the second statement, an easy calculation gives
\begin{equation}\label{TecPlem}
\begin{aligned}
J(T^{1/2} f_n)|_{a=0}& =T \left(\frac{1}{2}\|(-\Delta)^{\alpha/2}
f_n\|_{L^2}^2-\mu \iint_{\mathbb{R}^N} D_{p}( |f_n|, |f_n|)\ dxdy \right)\\
& + \left(T-T^{p}\right)\mu \iint_{\mathbb{R}^N} D_{p}( |f_n|, |f_n|)\ dxdy.
\end{aligned}
\end{equation}
By the first statement, we have $\mathbb{D}_{p}(f_n)\geq \delta.$ Since $p>1$ and $\mu >0,$ the
desired result follows
from \eqref{TecPlem}. The proof in the case $a>0$ is similar.
\end{proof}

\begin{lem}  \label{Lemvan}
Assume that the sequences $\{z_n\}_{n\geq 1}$ and $\{|(-\Delta)^{\alpha/2}z_n|\}_{n\geq 1}$ are bounded in $L^2(\mathbb{R}^N).$
If there is some $R>0$ satisfying
\begin{equation}
\lim_{n\to \infty} \left( \sup_{y \in \mathbb{R}^N} \int_{B_{R}(y)}|z_n(x)|^{2}\ dx\right) = 0,
\label{vanishcon}
\end{equation}
then
$
\Lim{n\to \infty}\|z_n\|_{L^q} = 0
$
for every $2<q<2N/(N-2\alpha).$
\end{lem}
\begin{proof}
The result is a version of Lemma I.1 of \cite{Lio}. See \cite{ABnls} for  a proof in
the case $\alpha=1$ (the same argument works for the case $0<\alpha<1$ with obvious modification).
\end{proof}

The idea behind the proof of relative compactness of minimizing sequence $\{f_n\}_{n\geq 1}$ of $J(\sigma)$ is that, we can employ
the concentration compactness principle to the
sequence of non-negative functions $\rho_n$ defined by $\rho_n=|f_n|^2.$
To do this, for $n=1, 2, \ldots$ and $R>0,$ consider
the associated concentration function
$M_n(R)$ defined by
\[
M_n(R)=\sup_{y\in \mathbb{R}^N}\int_{B_R(y)}\rho_n\ dx,
\]
where $B_R(x)$ stands for the $n$-ball of radius $R$ and
center $x\in \mathbb{R}^N.$ Suppose that evanescence of the energy minimizing $\{f_n\}_{n\geq 1}$ occurs, that is, for all $R>0,$
$\Lim{n\to \infty}M_n(R)=0$
up to a subsequence.
By Lemma~\ref{NegHP}, $\{|f_n|\}_{n\geq 1}$ is bounded. Lemma~\ref{Lemvan} then implies
that $\Lim{n\to \infty}\|f_n\|_{L^r}= 0$ for any
$2<r<\frac{2N}{N-2\alpha}.$ Since $2<\frac{2Np}{N+\beta}<\frac{2N}{N-2\alpha},$ it
follows from \eqref{Mbou3} that
\[
\mathbb{D}_p(f_n)=\int_{\mathbb{R}^N}\mathcal{I}_\beta^{p}(f_n)|f_n|^{p}\ dx\leq C \|f_n\|_{L^{2Np/(N+\beta)} }^{2p} \to 0,
\]
as $n\to \infty.$
Consequently,
we obtain that
\[
J(\sigma) = \lim_{n\to \infty}J(f_n)\geq \liminf_{n\to \infty}\int_{\mathbb{R}^N}|(-\Delta)^{\alpha/2}|^2\ dx\geq 0,
\]
which contradicts $J(\sigma) <0.$
Thus, we conclude from the
concentration compactness principle (see Lemma~I.1 of \cite{Lio}) that one of the remaining
two alternatives, namely dichotomy or compactness, is the only
option here.
In what follows, for every $\sigma>0$ and any minimizing sequence $\{f_n\}_{n\geq 1}\subset H^\alpha(\mathbb{R}^N)$ of $J(\sigma),$ we denote
\begin{equation}\label{LLdef}
L=\lim_{R\to \infty}\left( \lim_{n\to \infty} M_n(R)\right)\in [0, \sigma].
\end{equation}
Thus, we have established the following lemma.
\begin{lem}
If $\{f_n\}_{n\geq 1}\subset H^\alpha(\mathbb{R}^N)$ be any minimizing sequence for $J(\sigma)$ and $L$ be defined by \eqref{LLdef}, then $L>0.$
\end{lem}
The remaining two
possibilities are $L\in (0, \sigma)$ (dichotomy) and $L=\sigma$ (compactness). The next step toward the proof of the relative compactness of minimizing sequences is to show that, we must have $L=\sigma.$

\smallskip

Before we describe how energy minimizing
sequences $\{f_n\}_{n\geq 1}$ of $J(\sigma)$ would behave in the case when
$L \in (0, \sigma),$
we need the following result.
\begin{lem}\label{CEst}
Let $D^\alpha=(-\Delta)^{\alpha/2}, 0<\alpha<1.$
If $f,g\in \mathcal{S}(\mathbb{R}^N),$ then
\[
\|\left[D^\alpha,f\right]g\|_{L^2}\leq C_1\|\nabla f\|_{L^{r_1}}\|D^{\alpha-1}g\|_{L^{s_1}}+C_2\|D^\alpha f\|_{L^{r_2}}\|g\|_{L^{s_2}},
\]
where $\left[X, Y \right]=XY-YX$ represents the commutator, $s_1, s_2\in [2,\infty),$ and
\[
\frac{1}{r_1}+\frac{1}{s_1}=\frac{1}{r_2}+\frac{1}{s_2}=\frac{1}{2}.
\]
\end{lem}
\begin{proof}
This is a variant of the commutator estimate result of Kato and Ponce in \cite{Kato} and a proof is
given in \cite{Guo}.
\end{proof}

\begin{lem}\label{revLem}
Suppose that $f_n\in H^\alpha(\mathbb{R}^N), \|f_n\|_{L^2}^2\to \sigma, J(f_n)\to J(\sigma)$ and
$L\in (0,\sigma),$ where $L$ is as defined in \eqref{LLdef}.
For some subsequence of $\{f_n\}_{n\geq 1},$ which we continue to
denote by the same, the following are true: For every $\varepsilon>0,$ there exists
$n_0\in \mathbb{N}$ and sequences $\{v_n\}_{n\geq 1}$ and $\{w_n\}_{n\geq 1}$ in $H^\alpha(\mathbb{R}^N)$
such that for every $n\geq n_0,$

\smallskip

$1.$ ${\displaystyle  \int_{\mathbb{R}^N}|v_n(x)|^2\ dx \in (L-\varepsilon, L+\varepsilon) }$

\smallskip

$2.$ ${\displaystyle \int_{\mathbb{R}^N}|w_n(x)|^2\ dx \in (\sigma-L-\varepsilon, \sigma-L+\varepsilon) }$

\smallskip

$3.$ ${\displaystyle J(f_n)\geq J(v_n)+J(w_n)-\varepsilon }.$
\end{lem}
\begin{proof}
Given $\varepsilon>0,$ by definition of $L,$ there exists $R_0$ such
that if $R>R_0,$ then $L-\varepsilon<\Lim{n\to \infty}M_n(R)\leq L.$ Therefore, after extracting
a subsequence of $\{M_n\}$ if needed, we can say that there exist $n_0\in \mathbb{N}$ such that for all $n\geq n_0,$
\[
L-\varepsilon<M_n(R)\leq M_n(2R)<L+\varepsilon.
\]
It then follows that for every $n\geq n_0,$ there exists $y_n\in \mathbb{R}^N$ such that
\begin{equation}\label{GeqLem}
L-\varepsilon <\int_{B_{R}(y_n)}\rho_n\ dx \leq \int_{B_{2R}(y_n)}\rho_n\ dx <L+\varepsilon.
\end{equation}
Now introduce smooth cut-off functions $\phi$ and $\psi,$ defined on $\mathbb{R}^N,$ such that
$\phi(x) \equiv 1$ for $|x|\leq 1;$ $\phi(x)\equiv 0$ for $|x|\geq 2;$ $\psi(x)\equiv 1$ for $|x|\geq 2$; and $\psi(x)\equiv 0$ for $|x|\leq 1.$
Denote by $\phi_R$ and $\psi_R$ the functions $\phi_R(x)=\phi(x/R)$ and $\psi_R(x)=\psi(x/R),$ respectively.
Define $v_n(x)=\phi_R(x-y_n)f_n(x)$ and $w_n(x)=\psi_R(x-y_n)f_n(x).$ With these definitions together with \eqref{GeqLem},
Statements 1 and 2 are clear. To prove Statement 3, using Lemma~\ref{CEst} with $r_1=\infty, s_1=2,$ and $s_2=1+\frac{4\alpha}{N},$ one obtains that
\[
\begin{aligned}
\int_{\mathbb{R}^N}|[(-\Delta)^{\alpha/2},\widetilde{\phi}_R]f_n|^2\ dx& \leq C_1\|\nabla \widetilde{\phi}_R\|_{\infty}\|(-\Delta)^{(\alpha-1)/2}f_n\|_{L^2}\\
&+ C_2 R^{\alpha(N-2\alpha)/(N+2\alpha) } \|(-\Delta)^{\alpha/2}\widetilde{\phi}\|_{L^{2(N+2\alpha)/(2\alpha-N)}}\|f_n\|_{H^\alpha},
\end{aligned}
\]
where, for ease of notation, we denote $\widetilde{\phi}_R(x)=\phi_R(x-y_n)$ for $x\in \mathbb{R}^N.$
Using the fact $\|\nabla \widetilde{\phi}_{R}\|_\infty=\|\nabla \widetilde{\phi}\|_\infty / R,$ it immediately
follows from what we have just obtained that
\begin{equation}\label{CEst2}
\int_{\mathbb{R}^N}|(-\Delta)^{\alpha/2}(\widetilde{\phi}_R f_n)|^2\ dx\leq \int_{\mathbb{R}^N}(\widetilde{\phi}_R)^2|(-\Delta)^{\alpha/2} f_n|^2\ dx +C\varepsilon
\end{equation}
for sufficiently large $R.$ Similarly, we have the following estimate
\begin{equation}\label{CEst7}
\int_{\mathbb{R}^N}|(-\Delta)^{\alpha/2}(\widetilde{\psi}_R f_n)|^2\ dx\leq \int_{\mathbb{R}^N}(\widetilde{\psi}_R)^2|(-\Delta)^{\alpha/2} f_n|^2\ dx +C\varepsilon,
\end{equation}
where $\widetilde{\psi}_R(x)=\psi_R(x-y_n)$ for $x\in \mathbb{R}^N.$
Taking into account of $\phi_R^2+\psi_R^2 \equiv 1$ on $\mathbb{R}^N$ and using the estimates \eqref{CEst2} and \eqref{CEst7}, a direct
computation yields
\begin{equation}\label{CEst14}
\begin{aligned}
J(v_n)& +J(w_n)  \leq b \int_{\mathbb{R}^N}\left( \widetilde{\phi}_R^2(x)+\widetilde{\psi}_R^2(x)\right)|f_n|^s\ dx-b \int_{\mathbb{R}^N}|\widetilde{\phi}_Rf_n|^s\ dx \\
& -b \int_{\mathbb{R}^N}|\widetilde{\psi}_Rf_n|^s\ dx + J(f_n)+
\mu \iint_{\mathbb{R}^N}\left( \widetilde{\phi}_R^2(x)+\widetilde{\psi}_R^2(x)\right) D_p(|f_n|, |f_n|)\ dxdy\\
& - \mu \iint_{\mathbb{R}^N}D_p(|\widetilde{\phi}_Rf_n|, |\widetilde{\phi}_Rf_n|)dxdy-\mu \iint_{\mathbb{R}^N}D_p(|\widetilde{\psi}_Rf_n|, |\widetilde{\psi}_Rf_n|)dxdy.
\end{aligned}
\end{equation}
Let us denote by $\mathbf{M}_b^s[f_n]$ the first three
integrals on the right-hand side of \eqref{CEst14} and by $\mathbf{F}_\mu^p[f_n]$ the last three integrals
on the right-hand side of \eqref{CEst14}.
Now let $\mathbb{B}_R=B(y_k,2R)-B(y_k,R)$, $2_{\alpha}^\ast=2N/(N-2\alpha)$, and $\lambda^\prime=N(s-2)/2\alpha s.$
Then, it follows from the Interpolation inequality that
\begin{equation}\label{Zinqu1}
\begin{aligned}
 \mathbf{M}_b^s[f_n]& \leq b \int_{\mathbb{R}^N}|\widetilde{\phi}_R^2-|\widetilde{\phi}_R|^s ||f_n|^s\ dx+ b \int_{\mathbb{R}^N}|\widetilde{\psi}_R^2-|\widetilde{\psi}_R|^s ||f_n|^s\ dx \\
 & \leq 4b \int_{\mathbb{B}}|f_n|^s\ dx\leq C \|f_n\|_{L^2(\mathbb{B}_R)}^{(1-\lambda^\prime)s }
 \|f_n\|_{L^{2_\alpha^\ast}(\mathbb{B}_R)}^{\lambda^\prime s }\leq C \|f_n\|_{H^\alpha(\mathbb{B}_R)}^{\lambda^\prime s}\|f_n\|_{L^2(\mathbb{B}_R)}^{(1-\lambda^\prime)s }\leq C\varepsilon.
\end{aligned}
\end{equation}
Similarly, if
$\mu^\prime=(Np-N-\beta)/2\alpha p,$ then it follows from
\eqref{Mbou3} that
\begin{equation}\label{Zinqu}
\mathbf{F}_\mu^p[f_n] \leq C \|f_n\|_{H^\alpha(\mathbb{B}_R)}^{2p \mu^\prime}\|f_n\|_{L^2(\mathbb{B}_R)}^{2p(1-\mu^\prime)}\leq C\varepsilon,
\end{equation}
where $C$ is independent of $R$ and $n.$ Finally, taking into account
of the estimates \eqref{Zinqu1} and \eqref{Zinqu}, Statement 3 follows from \eqref{CEst14}.
\end{proof}

\begin{lem}\label{revInsubadd}
If $0<L<\sigma,$ then $J(\sigma)\geq J(L)+J(\sigma-L).$
\end{lem}
\begin{proof}
First observe that if a function $v$ satisfies $\left\vert \int_{\mathbb{R}^N}|v|^2\ dx- L \right\vert<\varepsilon,$ then we have that
$\int_{\mathbb{R}^N}|\eta v|^2\ dx= L,$ where $\eta=(L/\|v\|_{L^2}^2)^{1/2}$ satisfies $|\eta-1|<A_1\varepsilon$
with $A_1>0$ independent of $v$ and $\varepsilon.$ Thus
\[
J(\sigma) \leq J(\eta v)\leq J(v)+A_2\varepsilon,
\]
where the constant $A_2$ depends only on $A_1$ and the power $\|v\|_{L^2}^2.$ A similar estimate
holds for the function $w$ such that
\[
\left\vert \int_{\mathbb{R}^N}|w|^2\ dx- (\sigma-L) \right\vert<\varepsilon.
\]
Taking into account of these observations and Lemma~\ref{revLem}, it follows immediately
that there exists a subsequence $\{f_{n_k}\}_{k\geq 1}$ of $\{f_n\}_{n\geq 1}$ and
corresponding functions $v_{n_k}$ and $w_{n_k}$ for $k=1,2, \ldots$ such that for all $k,$
\[
\begin{aligned}
  J(v_{n_k})& \geq J(L) - \frac{1}{k},\
 J(w_{n_k})\geq J(\sigma-L)-\frac{1}{k},\ \mathrm{and} \\
 & J(f_{n_k})\geq J(v_{n_k})+J(w_{n_k})-\frac{1}{k},
 \end{aligned}
\]
and
so
\begin{equation}\label{revinqE}
J(v_{n_k})\geq J(L) + J(\sigma-L)-\frac{1}{k}.
\end{equation}
Passing the limit as $k\to \infty$ on both sides of \eqref{revinqE} yields the desired result.
\end{proof}

\begin{lem}\label{singleDi}
If $\{f_n\}$ be any minimizing sequence of $J(\sigma),$ then $L\not\in (0, \sigma).$
\end{lem}
\begin{proof}
Suppose to the contrary that $L\in (0, \sigma).$ Let $\sigma^\prime=L$ and $\sigma^{\prime \prime}=\sigma-L.$
Then, $\sigma=\sigma^\prime+\sigma^{\prime \prime}.$
Lemma~\ref{revInsubadd} implies that
$J(\sigma)\geq J(L)+J(\sigma-L),$ which gives
\begin{equation}\label{subaddSHT}
J(\sigma^\prime+\sigma^{\prime \prime})\geq J(\sigma^\prime)+J(\sigma^{\prime \prime}).
\end{equation}
To deduce a contradiction, we now claim the function $\sigma\mapsto J(\sigma)$ is
strictly subadditive, i.e., for all $\sigma^\prime>0$ and $\sigma^{\prime \prime}>0,$
\begin{equation}\label{subSING}
J(\sigma^\prime+\sigma^{\prime \prime})< J(\sigma^\prime)+J(\sigma^{\prime \prime}).
\end{equation}

\begin{remark}
All results proved above remain true for the problem $\widetilde{J}(\sigma).$
The proof of \eqref{subSING} below differs from the original ideas developed in \cite{Lio, Lio2}.
The advantage of this technique is that the same argument goes through unchanged
to prove an analogue strict inequality for the problem
$\widetilde{J}(\sigma)$ related to equations with an arbitrary number of combined
nonlinearities.
\end{remark}
To see \eqref{subSING}, let $\{z_n\}_{n\geq 1}$ and $\{w_n\}_{n\geq 1}$ be sequences of functions in $H^\alpha(\mathbb{R}^N)$ such that
\[
\|z_n\|_{L^2}^2\to \sigma^\prime,\ J(z_n)\to J(\sigma^{\prime}),\ \|w_n\|_{L^2}^2\to \sigma^{\prime \prime},\ J(w_n)\to J(\sigma^{\prime \prime})
\]
as $n\to \infty.$ Consider the sequence $\{(K_1^n, K_2^n)\}_{n\geq 1}\subset \mathbb{R}^2$ defined by
\[
\begin{aligned}
& K_1^n=\frac{1}{\|z_n\|_{L^2}^2}\left( \frac{1}{2}\int_{\mathbb{R}^N}|(-\Delta)^{\alpha/2} z_n|^2\ dx -b\|z_n\|_{L^s}^s- \mu \mathbb{D}_{p}(z_n)\right), \\
& K_2^n= \frac{1}{\|w_n\|_{L^2}^2}\left(\frac{1}{2}\int_{\mathbb{R}^N}|(-\Delta)^{\alpha/2} w_n|^2\ dx - b\|w_n\|_{L^s}^s- \mu \mathbb{D}_{p}(w_n) \right).
\end{aligned}
\]
By passing to a subsequence, we can assume that $(K_1^n, K_2^n)\to (K_1,K_2)$ in $\mathbb{R}^2.$ Three
cases are possible: $K_1<K_2, K_1>K_2,$ and $K_1=K_2.$
Suppose first that $K_1<K_2.$ Define a sequence $\{F_n\}_{n\geq 1}$ in $H^\alpha(\mathbb{R}^N)$ by
$F_n=T^{1/2}z_n,$
where $T=(\sigma^\prime+\sigma^{\prime \prime})/\sigma^{\prime}.$
Then $F_n\to \sigma^\prime+\sigma^{\prime \prime}$ and consequently, $J(\sigma^{\prime}+\sigma^{\prime \prime})\leq \Lim{n\to \infty}J(F_n).$
A straightforward calculation gives
\begin{equation}\label{stAddIn}
J(F_n)=\frac{T}{2} \|(-\Delta)^{\alpha /2}z_n\|_{L^2}^2-bT^{s/2}\|z_n\|_{L^s}^s- \mu T^{p}\iint_{\mathbb{R}^N} D_{p}(|z_n|, |z_n|)\ dxdy.
\end{equation}
Since $T>1$, $p-1>0$, $s-2>0,$ and $a, \lambda$ are nonnegative with $a+\lambda>0,$ it follows from \eqref{stAddIn}
that $J(F_n)\leq T J(z_n).$ This then implies that
\[
J(\sigma^{\prime}+\sigma^{\prime \prime})\leq T \lim_{n\to \infty}J(z_n)=T \sigma^\prime K_1.
\]
Put $\delta=\sigma^{\prime \prime}(K_2-K_1).$ Since $K_1<K_2,$ we have $\delta>0.$ It follows from the inequality
we have just obtained that
$ J(\sigma^\prime+\sigma^{\prime \prime})\leq \sigma^{\prime}K_1+\sigma^{\prime \prime}K_2-\delta.$ Thus, we obtain the
strict inequality, $J(\sigma^\prime+\sigma_2)<J(\sigma^\prime)+J(\sigma^{\prime \prime}),$
 which contradicts \eqref{subaddSHT}.

\smallskip

Using the similar argument as in the case $K_1<K_2$, the case $K_1>K_2$ also leads to a contradiction and will not be repeated.
Finally, consider the case $K_1=K_2.$ As in
the preceding paragraph, define $F_n=T^{1/2}z_n.$
If both $a$ and $\lambda$ are positive or $a=0$ and $\lambda>0,$
using Lemma~\ref{PoLem1}, there exists $\delta>0$ such that
\[
\frac{1}{2}\|(-\Delta)^{\alpha/2}F_n\|_{L^2}^2-\mu \mathbb{D}_p(F_n)\leq T \left( \frac{1}{2}\|(-\Delta)^{\alpha/2}z_n\|_{L^2}^2-\mu \mathbb{D}_p(z_n)\right)-\delta.
\]
for sufficiently large $n.$ This in turn implies that
\begin{equation}\label{extraIne1}
J(F_n)\leq T J(z_n)|_{a=0}-b T^{s/2}\|z_n\|_{L^s}^s-\delta,
\end{equation}
where $b=a/s.$ In this case, since $T>1, b\geq 0, \lambda>0, $ and $s-2>0,$
it follows from \eqref{extraIne1} that $J(F_n)\leq T J(z_n)-\delta$
for sufficiently large $n.$
Thus, we obtain that
\begin{equation}\label{saddLast}
J(\sigma^{\prime }+\sigma^{\prime \prime})\leq  T \lim_{n\to \infty}J(z_n)-\delta = T K_1\sigma^{\prime} -\delta=K_1\sigma^{\prime}+\sigma^{\prime \prime}K_1-\delta
\end{equation}
Since the equality $K_1=K_2$ holds, \eqref{saddLast} gives the desired contradiction.
If $a>0$ and $\lambda=0,$ then making use of Lemma~\ref{PoLem1} again, there exists $\delta>0$ such that
\[
\frac{1}{2}\|(-\Delta)^{\alpha/2}F_n\|_{L^2}^2-b \|F_n\|_{L^s}^s \leq T \left( \frac{1}{2}\|(-\Delta)^{\alpha/2}z_n\|_{L^2}^2-b \|z_n\|_{L^s}^s\right)-\delta.
\]
for sufficiently large $n.$ Then, we again obtain that $J(F_n)\leq T J(z_n)-\delta$ for
sufficiently large $n$ and \eqref{saddLast} gives the desired contradiction.
\end{proof}

\begin{lem}\label{shiftLem}
Suppose the case $L=\sigma.$ Then there exists $\{y_n\}\subset \mathbb{R}^N$ such that
\begin{itemize}
\item[$1.$] for any $\Lambda<\sigma$ there exists a number $R=R(\Lambda)>0$ and $n_0=n_0(\Lambda)\in \mathbb{N}$ such that for all $n\geq n_0,$
\begin{equation}\label{ExiIn1}
\int_{B_R(y_n)}|f_n|^2\ dx>\Lambda.
\end{equation}
\item[$2.$] the shifted sequence $\widetilde{f}_n(x)=f_n(x-y_n)$, $x\in \mathbb{R}^N,$ converges
(up to a subsequence) in $H^\alpha(\mathbb{R}^N)$ to some $u\in P_\sigma.$
\end{itemize}
\end{lem}
\begin{proof}
Statement 1 is standard; we include the proof here for the readers convenience.
Since $L=\sigma,$ there exists $R_0$ and $n_0(R_0)\in \mathbb{N}$ such that for all $n\geq n_0(R_0),$ one has
$
M_n(R_0)>\sigma/2.
$
Consequently, by the definition of $M_n,$ we can find $\{y_n\}\subset \mathbb{R}^N$ such that
$
\|f_n\|_{L^2(B_{R_0}(y_n) ) }^2>\sigma/2
$
for sufficiently large $n.$
Next let $\Lambda<\sigma$ be given. We may assume that $\Lambda>\frac{\sigma}{2}.$
Since $L=\sigma,$ we can find a number $R_0(\Lambda)$ and $n_0(\Lambda)\in \mathbb{N}$ such that if $n\geq n_0(\Lambda),$ then
\begin{equation}\label{ExiIn2}
\int_{B_{R_0(\Lambda)}(y_n(\Lambda))}|f_n|^2\ dx >\Lambda
\end{equation}
for some point $y_n(\Lambda)\in \mathbb{R}^N.$ Since the power $\|f_n\|_{L^2}^2=\sigma$ for each $n,$ it follows
that $B_{R_0}(y_n)\cap B_{R_0(\Lambda)}(y_n(\Lambda))\neq \emptyset,$ i.e.,
$|y_n(\Lambda)-y_n|\leq R_0+R_0(\Lambda)$ for $\Lambda>\sigma/2.$
 Now define $R=R(\Lambda)=2R_0(\Lambda)+R_0,$ then
we have that $B_{R_0(\Lambda)}(y_n(\Lambda))\subset B_R(y_n),$ and so, \eqref{ExiIn1} follows
from \eqref{ExiIn2} for all $n\geq n_0(\Lambda).$

\smallskip

Statement 1 now ensures that for all $k\in \mathbb{N},$ there exists $R_k$ such that for sufficiently large $n,$ we have that
$
\|\widetilde{f}_n\|_{L^2(B_{R_k}(0))}^2>L-1/k.
$
Since $\{\widetilde{f}_n\}$ is uniformly
bounded in $H^\alpha(\mathbb{R}^N),$ and therefore
also in $H^\alpha(\Omega)$ for any bounded domain $\Omega$ in $\mathbb{R}^N.$
Thus, from Rellich-type lemma (see for example, Corollary~7.2 of \cite{NPV}), it follows that $\{\widetilde{f}_n\}$
converges in $L^2(\Omega)$ norm (up to a subsequence) to some $u\in L^2(\Omega)$ satisfying
$
\|u\|_{L^2(B_{R_k}(0))}^2>L-1/k.
$
By applying the standard Cantor diagonalization argument together with the fact
that $\|\widetilde{f}_n\|_{L^2}^2=\sigma, \forall n,$ it can be shown that
up to a subsequence $\widetilde{f}_n \to u$ in $L^2(\mathbb{R}^N)$ satisfying $\|u\|_{L^2}^2=\sigma.$
We now write
\begin{equation}\label{conIn23}
\begin{aligned}
\mathbb{D}_{p}(\widetilde{f}_n)-\mathbb{D}_{p}(u) & =  \iint_{\mathbb{R}^N}\frac{|\widetilde{f}_n(x)|^{p}}{|x-y|^{N-\beta}}\left( |\widetilde{f}_n(y)|^{p}-|u(y)|^{p} \right) \ dxdy \\
& + \iint_{\mathbb{R}^N}\frac{|u(y)|^{p}}{|x-y|^{N-\beta} }\left( |\widetilde{f}_n(x)|^{p} -|u(x)|^{p} \right) \ dx dy.
\end{aligned}
\end{equation}
Applying
the Hardy-Littlewood-Sobolev inequality and using the
fact
that $\{\widetilde{f}_n\}_{n\geq 1}$ is bounded in $H^\alpha(\mathbb{R}^N),$ the first term on the
right-hand side of \eqref{conIn23} satisfies
\begin{equation}\label{ConvInq}
\left|\mathbf{I}_1\right\vert=\left\vert \iint_{\mathbb{R}^N}\frac{|\widetilde{f}_n(x)|^{p}}{|x-y|^{N-\beta}}\left( |\widetilde{f}_n(y)|^{p}-|u(y)|^{p} \right) \ dxdy \right\vert
\leq C \||\widetilde{f}_n|^p-|u|^p\|_{L^{2N/(N+\beta)}}.
\end{equation}
Next, for any two numbers $S$ and $T$ in $\mathbb{R},$ one has for $p\geq 1,$
\begin{equation}\label{ConvInq3}
\begin{aligned}
\left\vert |T|^{p-1}T-|S|^{p-1}S \right\vert& =\left\vert \int_{0}^{1}\frac{d}{dt}|S+t(T-S)|^{p-1}(S+t(T-S))\ dt \right\vert \\
& \leq p|T-S|\int_{0}^{1}|S+t(T-S)|^{p-1}\ dt \\
& \leq p|T-S|\int_{0}^{1} \left(t|T|^{p-1}+(1-t)|S|^{p-1} \right)\ dt\\
& =\frac{p}{2}|T-S|\left(|T|^{p-1}+|S|^{p-1} \right).
\end{aligned}
\end{equation}
Using \eqref{ConvInq3} into \eqref{ConvInq} and applying H\"{o}lder inequality, we obtain that
\[
\begin{aligned}
\left\vert \mathbf{I}_1\right\vert& \leq C \left(\int_{\mathbb{R}^N}\left(|\widetilde{f}_n|^{p-1}+|u|^{p-1} \right)^{\frac{2N}{N+\beta}}|\widetilde{f}_n-u|^{ \frac{2N}{N+\beta}}\ dx \right)^{\frac{N+\beta}{2N} }\\
& \leq C \left( \int_{\mathbb{R}^N}\left(|\widetilde{f}_n|^{ \frac{2Np}{N+\beta}}+|u|^{\frac{2Np}{N+\beta} } \right)\ dx \right)^{\rho} \|\widetilde{f}_n-u \|_{L^{\frac{2Np}{N+\beta} }}\leq C \|\widetilde{f}_n-u \|_{L^{\frac{2Np}{N+\beta} }}
\end{aligned}
\]
where $\rho= \frac{N+\beta}{2N}\left(1-\frac{1}{p}\right).$
Now, using the standard Interpolation inequality and the fractional Sobolev inequality, it follows that
\begin{equation}\label{ConvInq5}
\begin{aligned}
\left\vert \mathbf{I}_1 \right\vert
 \leq C \|\widetilde{f}_n-u \|_{L^2}^{\lambda^\prime} \|\widetilde{f}_n-u \|_{L^{\frac{2N}{N-2\alpha} } }^{1-\lambda^\prime}\leq C \|\widetilde{f}_n-u \|_{L^2}^{\lambda^\prime},
\end{aligned}
\end{equation}
where $\lambda^\prime=(N+\beta-Np+2p\alpha)/2p\alpha.$
The right-hand side of \eqref{ConvInq5} goes to zero since $\widetilde{f}_n\to u$ in $L^2.$
Similarly, the second term on the right-hand side of \eqref{conIn23} goes to zero.
Thus, we have that
$\Lim{n\to \infty}\mathbb{D}_{p}(\widetilde{f}_n)=\mathbb{D}_{p}(u).$ Using another application of the Interpolation inequality, one obtains that
\[
\|\widetilde{f}_n-u\|_{L^{s}}\leq \|\widetilde{f}_n-u\|_{L^{2}}^{\lambda^{\prime \prime} } \|\widetilde{f}_n-u\|_{L^{2N/(N-2\alpha)}}^{1-\lambda^{\prime \prime}}\leq C \|\widetilde{f}_n-u\|_{L^{2}}^{\lambda^{\prime \prime}},
\]
where $\lambda^{\prime \prime}=(2N-sN+2\alpha s)/2\alpha s.$ Thus, $\Lim{n\to \infty}\|\widetilde{f}_n\|_{L^s}=\|u\|_{L^s}.$
Furthermore, as a consequence of the weak lower semi-continuity of the norm in a Hilbert space, we can assume,
by extracting another
subsequence if necessary, that $\widetilde{f}_n\rightharpoonup u$ weakly in $H^\alpha,$ and that
\[
\|u\|_{H^\alpha}\leq \liminf_{n\to \infty}\|\widetilde{f}_n\|_{H^\alpha}.
\]
It follows then that
\[
J(u)\leq \lim_{n\to \infty}J(\widetilde{f}_n)=J(\sigma),
\]
and since $\widetilde{f}_n\to f$ in $L^2(\mathbb{R}^N),$ we also
have that $\|u\|_{L^2}^2=\Lim{n\to \infty}\|\widetilde{f}_n\|_{L^2}^2=\sigma.$
By the definition of the infimum $J(\sigma),$ we must have $J(u)=J(\sigma)$ and $u\in \Sigma_\sigma.$ Finally, the facts
$J(u)=\Lim{n\to \infty}J(\widetilde{f}_n)$, $\|u\|_{L^s}=\Lim{n\to \infty}\|\widetilde{f}_n\|_{L^s}$,
$\mathbb{D}_{p}(u)=\Lim{n\to \infty}\mathbb{D}_{p}(\widetilde{f}_n),$ and
$\|u\|_{L^2}=\Lim{n\to \infty}\|\widetilde{f}_n\|_{L^2}$ together
imply that $\|u\|_{H^\alpha}=\Lim{n\to \infty}\|\widetilde{f}_n\|_{H^\alpha},$ and
from a standard exercise in the elementary Hilbert space theory one then obtains that $\widetilde{f}_n\to u$ in $H^\alpha$ norm.
\end{proof}

We can now prove our main results. Since we ruled out the cases $L=0$ and $L\in (0, \sigma),$ the
only option for any minimizing sequence of $J(\sigma)$ is the compactness, i.e., $L=\sigma.$
Hence, by Lemma~\ref{shiftLem}, the set $P_\sigma$ is
nonempty and Statement 1 of Theorem~\ref{Mainpower} holds.

\smallskip

Next suppose that $u\in H^\alpha(\mathbb{R}^N)$ is a minimizer of $J(\sigma),$ that is, $\|u\|_{L^2}^2=\sigma$ and $J(u)=J(\sigma).$
Then it satisfies the Euler-Lagrange differential equation
\begin{equation}\label{lagra3}
(-\Delta)^\alpha u-a|u|^{s-2}u-\lambda \left( |x|^{\beta-N}\star |u|^p \right) |u|^{p-2}u=-\omega u
\end{equation}
for some $\omega\in \mathbb{R}.$ This gives that
\begin{equation}\label{lagra4}
\|(-\Delta)^{\alpha/2}\|_{L^2}^2-a\|u\|_{L^s}^s-\lambda \mathbb{D}_p(u)=-\omega \|u\|_{L^2}^2.
\end{equation}
We know that $J(u)<0.$ Since $a=bs>b$ and $\lambda=2\mu p>\mu,$ it follows that the
left side of \eqref{lagra4} is negative and so, $\omega>0.$

\smallskip

We now show that $|u|, |u|^\ast\in P_\sigma$ (for definition and properties about
symmetric rearrangements, see, for example, Chapter 3 of \cite{Lieb}).
Using the fact
\begin{equation}\label{symPO2}
u\in H^\alpha(\mathbb{R}^N)\Rightarrow |u|\in H^\alpha(\mathbb{R}^N),\ \|(-\Delta)^{\alpha/2}|u|\|_{L^2}\leq \|(-\Delta)^{\alpha/2}u\|_{L^2},
\end{equation}
it follows that $J(|u|)\leq J(u).$
Thus, $P_\sigma$ also contains $|u|$ and hence, the minimizer $u$ can be chosen to be $\mathbb{R}$-valued.
To prove $|u|^\ast\in P_\sigma,$ we need the following fact
\begin{equation}\label{symPO1}
u\in H^\alpha(\mathbb{R}^N)\Rightarrow |u|^\ast \in H^\alpha(\mathbb{R}^N),\ \|(-\Delta)^{\alpha/2}|u|^\ast\|_{L^2}\leq \|(-\Delta)^{\alpha/2}|u|\|_{L^2}
\end{equation}
This is proved in Lemmas~3.4 and 3.5 of \cite{ABS} for $\alpha=1/2$ and such a proof for the
case $0<\alpha<1$ can be constructed by adapting the same argument.
Moreover, it is well-known that the symmetric rearrangement preserve the $L^p$ norm, i.e.,
\begin{equation}\label{LpSym}
\||f|^\ast\|_{L^r} =\|f\|_{L^r},\ 1\leq r \leq \infty.
\end{equation}
Furthermore, a classical rearrangement inequality
of F. Riesz-S.~L.~Sobolev (see for example, Theorem~3.7 of \cite{Lieb}) gives
\begin{equation}\label{symPO2}
\iint_{\mathbb{R}^N}D_p(|u|^\ast, |u|^\ast )\ dxdy
\geq \iint_{\mathbb{R}^N}D_p(|u|, |u| )\ dxdy .
\end{equation}
Taking into account of \eqref{symPO1}, \eqref{LpSym}, and \eqref{symPO2}, it follows that
\[
\||f|^\ast \|_{L^2}^2=\|f\|_{L^2}^2\ \ \mathrm{and}\ J(|f|^\ast)\leq J(f),\ \forall f\in H^\alpha(\mathbb{R}^N),
\]
which shows that $P_\sigma$ contains $|u|^\ast$ whenever it does $u.$

\smallskip

To show that $|u|>0$ on $\mathbb{R}^N,$ observe that $\widetilde{u}=|u|\in \Sigma_\sigma$ satisfies the
Euler-Lagrange differential equation
\begin{equation}\label{lagra9}
(-\Delta)^\alpha \widetilde{u} +\omega \widetilde{u}=f(\widetilde{u}),\ \mathrm{where}\ f(\widetilde{u})=a|\widetilde{u}|^{s-2}\widetilde{u}+\lambda ( |x|^{\beta-N}\star |\widetilde{u} |^p)|\widetilde{u}|^{p-2}\widetilde{u}.
\end{equation}
The Lagrange multiplier in \eqref{lagra9} stays same because it is determined by \eqref{lagra3}.
Since $\omega>0,$ we have the convolution formula
\[
\widetilde{u}=\mathcal{W}_\omega \star f(\widetilde{u})=\int_{\mathbb{R}^N}\mathcal{W}_\omega(x-x^\prime)f(\widetilde{u})(x^\prime)\ dx^\prime,
\]
where for any $\tau>0,$ the Bessel kernel $\mathcal{W}_\tau(x)$ is  given by
\begin{equation}\label{lagra11}
\mathcal{W}_\tau(x)=\mathcal{F}^{-1}\left( \frac{1}{\tau +|\xi|^{2\alpha}} \right).
\end{equation}
Since $\widetilde{u}$ is the
convolution of $\mathcal{W}_\omega$ with the function $f(\widetilde{u})$ which is nonnegative and not identically
zero, it follows that $\widetilde{u}>0$ on $\mathbb{R}^N.$

\smallskip

To prove Statement 2 of Theorem~\ref{Mainpower}, suppose it does not hold. Then there would
exist a subsequence $\{f_{n_k}\}$ of $\{f_n\}$ and $\varepsilon_0>0$ such that
\[
\inf_{u\in P_\sigma}\inf_{y\in \mathbb{R}^N}\|f_{n_k}(\cdot + y)-u\|_{H^\alpha} \geq \varepsilon_0
\]
for all $k\in \mathbb{N}.$ But since $\{f_{n_k}\}$ would itself enjoy being a
minimizing sequence for $J(\sigma).$ Consequently, there would exist $\{y_k\}\subset \mathbb{R}^N$ and $u_0\in P_\sigma$ such that
\[
\lim_{k\to \infty}\|f_{n_k}(\cdot + y_k)-u_0\|_{H^\alpha}=0,
\]
which is a contradiction and hence, Statement 2 follows.
Because of translation invariance of the functional $J(u)$ and the power $\int_{\mathbb{R}^N}|u|^2\ dx,$
Statement 3 is an immediate consequence of Statement 2.

\smallskip

To prove the stability of the set $P_\sigma,$ suppose the contrary. Then there would exist $\varepsilon_0>0,$ a sequence $\{v_n\}\subset H^\alpha(\mathbb{R}^N),$ and times $t_n$ enjoying
\begin{equation}
\inf_{h\in P_\sigma}\|v_n - h\|_{H^\alpha}<\frac{1}{n}
\end{equation}
and
\[
\inf_{h\in P_\sigma}\|u_n(\cdot, t_n)-h \|_{H^\alpha}\geq \varepsilon_0
\]
for all $n,$ where $u_n(x,t)$ solves the equation \eqref{NSpower} with $u_n(x,0)=v_n.$
Now taking into account of the convergence $v_n\to P_\sigma$ in $H^\alpha(\mathbb{R}^N),$
and $J(h)=J(\sigma), \|h\|_{L^2}^2=\sigma$ for $h\in P_\sigma,$ we would have
$J(v_n)\to J(\sigma)$ and $\|v_n\|_{L^2}^2\to \sigma.$ Take the numbers $\{\zeta_n\}\subset \mathbb{R}$ such that $\|\zeta_nv_n\|_{L^2}^2=\sigma$ for
all $n.$ Then $\zeta_n\to 1.$ Let us denote $w_n=\zeta_nu_n(\cdot, t_n).$ Then $w_n\in \Sigma_\sigma$ for each $n$ and
\[
\lim_{n\to \infty}J(w_n)=\lim_{n\to \infty}J(u_n(\cdot,t_n))=\lim_{n\to \infty}J(v_n)=J(\sigma).
\]
Thus, $\{w_n\}$ enjoys being a minimizing sequence of $J(\sigma).$ In consequence, by Theorem~\ref{Mainpower}, there would exists $h_n\in P_\sigma$ such that
$\|w_n-h_n\|_{H^\alpha}<\varepsilon_0/2.$ But then
\[
\begin{aligned}
\varepsilon_0\leq \|u_n(\cdot, t_n)-h_n \|_{H^\alpha}& \leq \|u_n(\cdot, t_n)-w_n \|_{H^\alpha}+\|w_n-h_n \|_{H^\alpha} \\
&\leq |1-\zeta_n|\|u_n(\cdot,t_n)\|_{H^\alpha}+\frac{\varepsilon_0}{2},
\end{aligned}
\]
which after passing the limit $n\to \infty$ yields $\varepsilon_0\leq \varepsilon_0/2,$ a contradiction.

\section{Standing waves for Choquard type systems }\label{coupledS}

In this section, we prove existence and stability of standing waves for the coupled fNLS system with Choquard-type nonlinearities
\begin{equation}\label{CfNLSp}
\left\{
\begin{aligned}
  & i\partial_t \Psi_1+(-\Delta)^\alpha \Psi_1= \lambda_1 \mathcal{I}_\beta^{p_1}(\Psi_1) |\Psi_1|^{p_1-2}\Psi_1 +c \mathcal{I}_\beta^{q}(\Psi_2) |\Psi_1|^{q-2}\Psi_1,\\
   & i\partial_t \Psi_2+(-\Delta)^\alpha \Psi_2= \lambda_2\mathcal{I}_\beta^{p_2}(\Psi_2) |\Psi_2|^{p_2-2}\Psi_2+c \mathcal{I}_\beta^{q}(\Psi_1) |\Psi_2|^{q-2}\Psi_2,
\end{aligned}
\right.
\end{equation}
where $(x,t)\in \mathbb{R}^N\times (0,\infty)$ and $\lambda_i, c>0.$ Throughout this section, we assume that the following conditions hold for
the powers $\alpha, \beta, p_1, p_2,$ and $q.$
\begin{equation}\label{Cfhypo1}
N\geq 2,\ 0<\alpha<1,\ \beta\in (0, N),\ 2\leq p_1, p_2, q < \frac{N+2\alpha+\beta}{N}.
\end{equation}
A standing wave solution of \eqref{CfNLSp} is a solution of the form
\[
(\Psi_1(x,t), \Psi_2(x,t))=(e^{-i\omega_1 t} u_1(x), e^{-i\omega_2 t} u_2(x))
\]
for some $\omega_1, \omega_2\in \mathbb{R}$
and $(u_1,u_2)$ solves the Choquard system \eqref{CfODEp}.
The associated
energy functional is
\[
E(u)=\frac{1}{2}\sum_{j=1}^2\|(-\Delta)^{\alpha/2}u_j\|_{L^2}^2-\sum_{j=1}^2\mu_j\mathbb{D}_{p_j}(u_j)-\mu \mathbb{D}_q(u_1,u_2),
\]
where $\mu_j=\lambda_j/2p_j$ for $j=1,2$ and $\mu=c/q.$ We look for the profile function $(u_1,u_2)$ in the
space $Y^\alpha(\mathbb{R}^N)=H^\alpha(\mathbb{R}^N)\times H^\alpha(\mathbb{R}^N)$
satisfying $\|u_1\|_{L^2}^2=\sigma_1$ and $\|u_2\|_{L^2}^2=\sigma_2$
for given $\sigma_1>0$ and $\sigma_2>0.$

\smallskip

To describe the main results of this section, let us first fix some definitions and notation.
We use the notations: $\mathbb{R}_{>0}=(0,\infty), \mathbb{R}_{\geq 0}=[0, \infty),$ and similar
meanings for $\mathbb{R}_{>0}^2$ and $\mathbb{R}_{\geq 0}^2.$
We write the ordered pairs in $\mathbb{R}^2$ as $\sigma=(\sigma_1,\sigma_2), \sigma^\prime=(\sigma_1^\prime, \sigma_2^\prime),$ etc.
Vectors in $Y^\alpha(\mathbb{R}^N)$ are written as $u=(u_1, u_2), f=(f_1,f_2),$ etc.
A sequence $\{f_n\}_{n\geq1}$ in $Y^\alpha(\mathbb{R}^N)$ is always understood as $f_n=(f_1^n, f_2^n).$

\smallskip

For any $\sigma\in \mathbb{R}_{>0}^2,$
denote by $\Sigma_{\sigma}=\Sigma_{\sigma_1}\times \Sigma_{\sigma_2}$ the product of $L^2(\mathbb{R}^N)$ spheres and
let $M_{\sigma}$ denotes the set of coupled standing wave solutions
\begin{equation}\label{cmDef}
M_{\sigma}=\left\{u\in Y^\alpha(\mathbb{R}^N): u\in \Sigma_{\sigma},\ E(u)=E(\sigma) \right\},\ E(\sigma):=\inf_{f\in \Sigma_{\sigma} }E(f).
\end{equation}
The following analogues of
Theorem~\ref{Mainpower} and its Corollary~\ref{Maincor} are the main results of this section.

\begin{thm}\label{CHMain}
Suppose that \eqref{Cfhypo1} holds. Then, for any $\sigma\in \mathbb{R}_{>0}^2,$
the set $M_{\sigma}$ defined in \eqref{cmDef}
is nonempty. Moreover, the following statements hold:

\smallskip

$1.$ If $f_n\in Y^\alpha(\mathbb{R}^N)$, $\|f_1^n\|_{L^2}^2\to \sigma_1$, $\|f_2^n\|_{L^2}^2\to \sigma_2,$ and $E(f_n)\to E(\sigma),$
then the sequence
$\{f_n\}_{n\geq 1}$ is relatively compact in $Y^\alpha(\mathbb{R}^N)$ up to a translation.

\smallskip

$2.$ $f_n \to M_{\sigma}$ in the following sense,
\[
\lim_{n\to \infty}\inf_{u\in M_{\sigma}}\|f_n-u\|_{Y^\alpha}=0.
\]
Furthermore, the solution set $M_{\sigma}$ has the following properties

\smallskip

$3.$ The Lagrange multiplier $(\omega_1, \omega_2)$ associated with $u=(u_1,u_2)$ on $\Sigma_\sigma$ satisfies $\omega_1>0$ and $\omega_2>0.$

\smallskip

$4.$ If $(u_1,u_2)\in M_{\sigma},$ then $(|u_1|, |u_2|)\in M_{\sigma}$ and $|u_1|>0, |u_2|>0$ on $\mathbb{R}^N.$ One also has $(|u_1|^\ast, |u_2|^\ast)\in M_{\sigma}$ whenever $(u_1,u_2)\in M_{\sigma}.$

\smallskip

\end{thm}
As an immediate consequence we can get the stability result:
\begin{cor}\label{CMaincor}
The set $M_{\sigma}$ is stable in the same sense as in Corollary~\ref{Maincor}.
\end{cor}
\begin{remark}
In Corollary~\ref{CMaincor}, we made the assumption that for any $(\Psi_1^0, \Psi_2^0)\in Y^\alpha(\mathbb{R}^N)$ and every $T>0,$ \eqref{CfNLSp} has a unique solution $(\Psi_1, \Psi_2)\in C([0,T], Y^\alpha(\mathbb{R}^N))$ with  $(\Psi_1(x,0), \Psi_2(x,0))=(\Psi_1^0(x), \Psi_2^0(x)).$
The map $(\Psi_1^0, \Psi_2^0)\mapsto (\Psi_1, \Psi_2)$ is locally Lipschitz from $Y^\alpha(\mathbb{R}^N)$ to $C([0,T], Y^\alpha(\mathbb{R}^N)).$
Moreover, the functional $E(\Psi_1(t), \Psi_2(t) )$ and the powers $\|\Psi_1(t) \|_{L^2}^2$, $\|\Psi_2(t) \|_{L^2}^2$ are
independent of $t\in [0,T].$
\end{remark}

\begin{remark}
In \eqref{AssPo}, \eqref{assCom}, and \eqref{Cfhypo1}, one can also include $\alpha=1,$ in which
case the same argument works and analogues
of Theorem~\ref{Mainpower}, Theorem~\ref{arbiTHM}, Theorem~\ref{CHMain}, and their corollaries remain true.
We fix $\alpha\in (0, 1)$ only to
avoid providing some additional technical details.
\end{remark}
In what follows we use the following notation
\begin{equation}\label{I1I2Def}
\begin{aligned}
& \mathbb{F}_1(f)=\frac{1}{2}\|(-\Delta)^{\alpha/2}f\|_{L^2}^2- \mu_1\mathbb{D}_{p_1}(f),\ \forall f\in H^\alpha(\mathbb{R}^N), \\
& \mathbb{F}_2(f)=\frac{1}{2}\|(-\Delta)^{\alpha/2}f\|_{L^2}^2- \mu_2\mathbb{D}_{p_2}(f),\ \forall f\in H^\alpha(\mathbb{R}^N).
\end{aligned}
\end{equation}
We will prove Theorem~\ref{CHMain} and its corollary following the same steps as used in the preceding section to prove
Theorem~\ref{Mainpower}.
As in the case of one parameter problem, we first prove some preliminaries lemmas.

\smallskip

Analogue of Lemma~\ref{NegHP} is the following.
\begin{lem}
Suppose the conditions \eqref{Cfhypo1} hold. Then, for any $\sigma\in \mathbb{R}_{>0}^2,$

\smallskip

(i) If $f_n\in Y^\alpha(\mathbb{R}^N), \|f_1^n\|_{L^2}^2\to \sigma_1$, $\|f_2^n\|_{L^2}^2\to \sigma_2,$
and $E(f_n)\to E(\sigma),$ then the sequence $\{f_n\}_{n\geq 1}$ is bounded in $Y^\alpha(\mathbb{R}^N).$

\smallskip

(i) One has $E(\sigma)\in (-\infty, 0).$
\end{lem}
\begin{proof}
To prove (i), we use the estimate \eqref{Mbou3} to obtain
\begin{equation}\label{crossC}
\begin{aligned}
& \mathbb{D}_{p_1}(f_1^n)=\int_{\mathbb{R}^N}\mathcal{I}_\beta^{p_1}(f_1^n)|f_1^n|^{p_1}\ dx \leq C \|f_n\|_{H^\alpha}^{2p_1\varrho_1}, \\
& \mathbb{D}_{p_2}(f_2^n)=\int_{\mathbb{R}^N}\mathcal{I}_\beta^{p_2}(f_2^n)|f_2^n|^{p_2}\ dx \leq C \|f_n\|_{H^\alpha}^{2p_2\varrho_2},
\end{aligned}
\end{equation}
where $\varrho_j=(Np_j-N-\beta)/2\alpha p_j$ for $j=1,2.$ Since $\|f_1^n\|_{L^2}$ and $\|f_2^n\|_{L^2}$ are
bounded, using the Hardy-Littlewood-Sobolev inequality, the Young's inequality, and the fractional Gagliardo-Nirenberg inequality, we have that
\begin{equation}\label{crossC2}
\begin{aligned}
\mathbb{D}_q(f_1^n, f_2^n)& \leq C \|f_1^n\|_{L^{2Nq/(N+\beta)}}^q \|f_2^n\|_{L^{2Nq/(N+\beta)}}^q  \\
& \leq \frac{C}{2}\left(\|f_1^n\|_{L^{2Nq/(N+\beta)}}^{2q} +\|f_2^n\|_{L^{2Nq/(N+\beta)}}^{2q} \right)\leq C\|f_n\|_{H^\alpha}^{2\varrho q},
\end{aligned}
\end{equation}
where $\varrho=(Nq-N-\beta)/2\alpha q.$
As in the proof of Lemma~\ref{NegHP}, we now write
\[
\frac{1}{2}\|f_n\|_{Y^\alpha}^2=E(f_n)+\sum_{j=1}^2 \mu_j\mathbb{D}_{p_j}(f_j^n)+\mu \mathbb{D}_q(f_n)+\frac{1}{2}\|f_1^n\|_{L^2}^2+\frac{1}{2}\|f_2^n\|_{L^2}^2.
\]
Since the sequence of numbers $\{E(f_n)\}_{n\geq 1}$ is bounded, using the estimates \eqref{crossC} and \eqref{crossC2}, we obtain that
\begin{equation}\label{crossC3}
\frac{1}{2}\|f_n\|_{Y^\alpha}^2\leq C \left(1+\|f_n\|_{H^\alpha}^{2p_1\varrho_1}+ \|f_n\|_{H^\alpha}^{2p_2\varrho_2}+\|f_n\|_{H^\alpha}^{2\varrho q} \right).
\end{equation}
Since $2p_j\varrho_j<2$ for $j=1,2$ and $2\varrho q<2,$ it follows from \eqref{crossC3} that the
sequence $\{f_n\}_{n\geq 1}$ is bounded in $Y^\alpha(\mathbb{R}^N).$
The statement $E(\sigma)>-\infty$ can be easily proved using the estimates \eqref{crossC} and \eqref{crossC2}.

\smallskip

To see $E(\sigma)<0,$ we use the fact that $\mu>0$ to obtain
\[
E(\sigma)\leq E(\sigma_1, 0)+E(0, \sigma_2).
\]
Since $\sigma_1, \sigma_2\in \mathbb{R}_{>0},$ as in Lemma~\ref{NegHP}, we have that $E(\sigma_1, 0)<0$ and $E(0, \sigma_2)<0$ and hence,
$E(\sigma)<0.$
\end{proof}
Analogue of Lemma~\ref{PoLem1} holds in the present context
without change of statement and an obvious modification in the proof. One also has the following

\begin{lem}\label{LMnegL}
For any $\sigma\in \mathbb{R}_{>0}^2,$ suppose that $f_n\in Y^\alpha(\mathbb{R}^N)$, $\|f_1^n\|_{L^2}^2\to \sigma_1$,
$\|f_2^n\|_{L^2}^2\to \sigma_2$, and $E(f_n)\to E(\sigma).$ Then there
exist numbers $\delta_1>0, \delta_2>0,$ and $n_0=n_0(\delta_j)\in \mathbb{N}$ such that for all $n\geq n_0,$
\[
\mathbb{F}_j(f_j^n)-\mu \mathbb{D}_q(f_n)\leq -\delta_j,\ j=1,2.
\]
\end{lem}
\begin{proof}
Suppose to the contrary that there is a sequence $\{f_n\}_{n\geq 1}$ in $Y^\alpha(\mathbb{R}^N)$ satisfying the hypotheses of the lemma and that
\[
\liminf_{n\to \infty}\left(\mathbb{F}_1(f_1^n)-\mu \mathbb{D}_q(f_n) \right)\geq 0.
\]
This implies that
\begin{equation}\label{LMnegL2}
E(\sigma)=\lim_{n\to \infty}E(f_n)\geq \liminf_{n\to \infty}\mathbb{F}_2(f_2^n).
\end{equation}
Using Lemma~\ref{Mainpower}, let $\widetilde{u}_2\in H^\alpha(\mathbb{R}^N)$ be such that
$\mathbb{F}_2(\widetilde{u}_2)=\Inf{f\in \Sigma_{\sigma_2}}\mathbb{F}_2(f).$ Then \eqref{LMnegL2} gives that
$
E(\sigma)\geq E(0, \widetilde{u}_2).
$
To deduce a contradiction, let $\widetilde{u}_1\in \Sigma_{\sigma_1}$ be such
that $\mathbb{F}_1(\widetilde{u}_1)-\mu \mathbb{D}_q(\widetilde{u}_1, \widetilde{u}_2)<0.$ Then it follows that
\[
E(\sigma)\leq E(\widetilde{u}_1, \widetilde{u}_2)=\mathbb{F}_1(\widetilde{u}_1)-\mu \mathbb{D}_q(\widetilde{u}_1, \widetilde{u}_2)+E(0, \widetilde{u}_2)<E(0, \widetilde{u}_2),
\]
which is a contradiction. The proof of the statement involving $f_2^n$ follows
the same lines and we will not repeat it.
\end{proof}
For any $\sigma\in \mathbb{R}_{>0}^2,$ let $\{(f_1^n, f_2^n)\}_{n\geq 1}$ be any minimizing sequence for $E(\sigma).$ We now employ
the concentration compactness principle to the sequence of non-negative
functions $\sigma_n=|f_1^n|^2+|f_2^n|^2.$ As in the preceding case,
consider the associated concentration function $P_n(R)$ defined by
\[
P_n(R)=\sup_{y\in \mathbb{R}^N}\int_{B_R(y)}\sigma_n\ dx,\ \mathrm{for}\ n=1, 2, \ldots,\ \mathrm{and}\ R>0.
\]
In what follows, we continue to denote $f_n=(f_1^n, f_2^n).$
Suppose that evanescence of the energy minimizing $\{f_n\}_{n\geq 1}$ occurs, that is,
$
\Lim{n\to \infty} P_n(R)=0
$
for all $R>0.$ Then, as before, we see from Lemma~\ref{Lemvan} that
$
\mathbb{D}_{p_j}(f_j^n)\leq C \|f_j^n\|_{L^{2Np_j/(N+\beta)} }^{2p_j}\to 0
$
as $n\to \infty$ for $j=1,2.$
Using \eqref{crossC2} and Lemma~\ref{Lemvan}, it also follows that $\mathbb{D}_q(f_n)\to 0$ as $n\to \infty.$
Then, we obtain that
\[
E(\sigma)=\lim_{n\to \infty}E(f_n)\geq \liminf_{n\to \infty}\int_{\mathbb{R}^N}\left( |(-\Delta)^{\alpha/2}f_1^n|^2+|(-\Delta)^{\alpha/2}f_2^n|^2\right)\ dx\geq 0,
\]
which contradicts $E(\sigma)<0.$
Let us denote
\begin{equation}\label{limitDEF}
M=\lim_{R\to \infty} \lim_{n\to \infty} \left( \sup_{y\in \mathbb{R}^N}\int_{B_R(y)}\sigma_n\ dx\right)\in [0, \sigma_1+\sigma_2].
\end{equation}
Thus, we have established the following lemma.
\begin{lem}
If $\{f_n\}\subset Y^\alpha(\mathbb{R}^N)$ be any minimizing sequence
for $E(\sigma)$ and $M$ be as defined in \eqref{limitDEF}, then $M>0.$
\end{lem}

\smallskip

Next, we rule out the possibility of the case $M\in (0, \sigma_1+\sigma_2).$ Analogue
of Lemma~\ref{revLem} is the following.
\begin{lem}\label{CrevLem}
For some subsequence of $\{f_n\}_{n\geq 1},$ which we
denote by the same, the following are true. For every $\varepsilon>0,$ there exists
$n_0\in \mathbb{N}$ and the sequences $\{(v_1^n, v_2^n)\}_{n\geq 1}$ and $\{(w_1^n, w_2^n)\}_{n\geq 1}$ of functions in $Y^\alpha(\mathbb{R}^N)$
such that for every $n\geq n_0,$

\smallskip

$1.$ ${\displaystyle \left\vert \int_{\mathbb{R}^N}\left( |v_1^n|^2+|v_2^n|^2\right)\ dx -M \right\vert <\varepsilon}$

\smallskip

$2.$ ${\displaystyle  \left\vert \int_{\mathbb{R}^N}\left( |w_1^n|^2+|w_2^n|^2\right)\ dx -((\sigma_1+\sigma_2)-M) \right\vert <\varepsilon }$

\smallskip

$3.$ ${\displaystyle E(f_n)\geq E(v_1^n, v_2^n)+E(w_1^n, w_2^n)-C \varepsilon }.$
\end{lem}
\begin{proof}
As in the proof of Lemma~\ref{revLem}, we choose $R\in \mathbb{R}$ and $n_0\in \mathbb{N}$ such that for all $n\geq n_0,$
\[
M-\varepsilon <P_n(R)\leq P_n(2R)<M+\varepsilon.
\]
Define the sequences $\{(v_1^n, v_2^n)\}$ and $\{(w_1^n, w_2^n)\}$ as follows
\[
(v_1^n, v_2^n)=\phi_R(x-y_n)f_n(x),\ \ (w_1^n, w_2^n)=\psi_R(x-y_n)f_n(x),\ x\in \mathbb{R}^N,
\]
where $\phi$ and $\psi$ are as in the proof of Lemma~\ref{revLem} and $y_n\in \mathbb{R}^N$ is chosen so that \eqref{GeqLem} holds with $\rho_n$
replaced by $\sigma_n$ and $L$ replaced by $M.$ Since the
sequences $\{v_1^n\}_{n\geq 1}$, $\{v_2^n\}_{n\geq 1}$, $\{w_1^n\}_{n\geq 1}$, and $\{w_2^n\}_{n\geq 1}$ are all bounded in $L^2(\mathbb{R}^N),$
so there exist $\sigma_1^\prime\in [0, \sigma_1]$ and $\sigma_2^{\prime}\in [0, \sigma_2]$ such that
up to a subsequence $\|v_1^n\|_{L^2}^2\to \sigma_1^\prime$ and $\|v_2^n\|_{L^2}^2\to \sigma_2^\prime$, whence one also has
$\|w_1^n\|_{L^2}^2\to \sigma_1-\sigma_1^\prime$ and $\|w_2^n\|_{L^2}^2\to \sigma_2-\sigma_2^\prime.$
Then, in view of \eqref{GeqLem}, Statements 1 and 2 of Lemma~\ref{CrevLem} follow.

\smallskip

Analogue of the inequality \eqref{CEst2} holds for $\widetilde{\phi}_Rf_1^n,$ $\widetilde{\phi}_Rf_2^n,$ $\widetilde{\psi}_Rf_1^n,$ and $\widetilde{\psi}_Rf_2^n.$
The inequality \eqref{CEst14} takes the form
\begin{equation}\label{CEst19}
E\left(v_n\right) +E\left(w_n\right)\leq
E\left(f_n\right) +\sum_{j=1}^2 \mathbf{F}_{\mu_j}^{p_j}[f_j^n]+\mathbf{F}_\mu^q[f_n] +C\varepsilon,
\end{equation}
where $\mathbf{F}_c^r[f]$ is as defined in the proof of Lemma~\ref{revLem}.
An estimate similar to \eqref{Zinqu} holds in the present context as well. Then, Statement 3 follows from \eqref{CEst19}.
\end{proof}

\begin{lem}\label{ECrev}
For every $\sigma\in \mathbb{R}_{>0}^2,$ there exists the numbers $\sigma_1^\prime\in [0, \sigma_1]$ and $\sigma_2^\prime\in [0, \sigma_2]$
such that $M=\sigma_1^\prime+\sigma_2^\prime$ and
\begin{equation}\label{ECrev71}
E(\sigma)\geq E(\sigma^\prime)+E(\sigma-\sigma^\prime).
\end{equation}
\end{lem}
\begin{proof}
With the notation of Lemma~\ref{CrevLem} and its proof,
for any $\varepsilon>0,$ since $\{E(v_n)\}_{n\geq 1}$ and $\{E(w_n)\}_{n\geq 1}$ are bounded, we can assume that $E(v_n)\to \Lambda_1$ and $E(w_n)\to \Lambda_2.$ Statement 3 of Lemma~\ref{CrevLem} then implies that
\[
\Lambda_1 +\Lambda_2 \leq E(\sigma)+C\epsilon.
\]
Next, for every $m\in \mathbb{N},$ choose the sequences $\{\left(v_1^{n,m},v_2^{n,m}\right)\}$ and $\{\left(w_1^{n,m},w_2^{n,m}\right)\}$ in
$Y^\alpha(\mathbb{R}^N)$ such that for $j=1,2,$
\[
\begin{aligned}
& \|v_j^{n,m}\|_{L^2}^2\to \sigma_j^\prime(m),\ E\left(v_1^{n,m},v_2^{n,m}\right) = \Lambda_1(m),\\
& \|w_j^{n,m}\|_{L^2}^2\to \sigma_j-\sigma_j^\prime(m),\ \mathrm{and}\  E\left(w_1^{n,m},w_2^{n,m}\right) = \Lambda_2(m),
\end{aligned}
\]
where $\sigma_1^\prime(m) \in [0,\sigma_1], \ \sigma_2^\prime(m) \in [0,\sigma_2],$
\begin{equation}\label{Limit1m}
|\sigma_1^\prime(m)+\sigma_2^\prime(m)-M| \leq \epsilon,\ \ \textrm{and}\ \ \Lambda_1(m) +\Lambda_2(m) \leq E(\sigma)+\frac{1}{m}.
\end{equation}
We may further extract a subsequence and assume that
$\sigma_1^\prime(m) \to \sigma_1^\prime \in [0,\sigma_1], \sigma_2^\prime(m) \to \sigma_2^\prime \in [0,\sigma_2], \Lambda_1(m) \to \Lambda_1,$
and $\Lambda_2(m)\to \Lambda_2.$ Moreover, after redefining
$v_n$ and $w_n$ to be diagonal entries
$v_n=\left(v_1^{n,n},v_2^{n,n}\right)$ and $w_n=\left(w_1^{n,n},w_2^{n,n}\right),$
we can say that $\|v_j^n\|_{L^2}^2\to \sigma_j^\prime$, $\|w_j^n\|_{L^2}^2\to \sigma_j-\sigma_j^\prime,$
$E\left(v_n\right)\to \Lambda_1,$ and $E(w_n)\to \Lambda_2.$
Now, letting $m$ tend to infinity on both sides of the first inequality in \eqref{Limit1m}, we deduce that $M=\sigma_1^\prime+\sigma_2^\prime.$
Next, we claim that
\begin{equation}\label{GeqInq57}
\Lambda_1 \geq E(\sigma^\prime)\ \ \textrm{and}\ \ \Lambda_2 \geq E(\sigma-\sigma^\prime).
\end{equation}
To see $\Lambda_1 \geq E(\sigma^\prime),$ suppose first that $\sigma_1^\prime>0$ and $\sigma_2^\prime>0.$
Define
\[
\beta_1^n=\frac{(\sigma_1^\prime)^{1/2}}{\|v_1^n\|_{L^2}}\ \ \textrm{and}\ \ \beta_2^n=\frac{(\sigma_2^\prime)^{1/2}}{\|v_2^n\|_{L^2}}.
\]
Then, we get $E\left(\beta_1^n v_1^n,\beta_2^n v_2^n\right)\geq E(\sigma^\prime).$ Since $\beta_1^n \to 1$ and $\beta_2^n \to 1$ as $n\to \infty,$
one has that $E\left(\beta_1^n v_1^n,\beta_2^n v_2^n\right)\to \Lambda_1$ and hence, $\Lambda_1 \geq E(\sigma^\prime).$
Suppose now that $\sigma_1^\prime=0$ and $\sigma_2^\prime>0$ (the same argument applies for the case $\sigma_1^\prime>0$ and $\sigma_2^\prime=0$).
Then, using the Hardy-Littlewood-Sobolev and
interpolation inequalities, one can see that $\mathbb{D}_{p_1}(v_1^n)\to 0$ and $\mathbb{D}_{q}(v_1^n, v_2^n)\to 0$ as $n\to \infty.$
Consequently, we get
\begin{equation*}
\Lambda_1  =\lim_{n\to \infty}E\left(v_n\right)
  = \lim_{n\to \infty}\left( \mathbb{F}_2\left(v_2^n\right)+ \int_{\mathbb{R}^N}|(-\Delta)^{\alpha/2}v_1^n|^2\ dx\right) \geq E(0,\sigma_2^\prime).
\end{equation*}
This finishes the proof of $\Lambda_1 \geq E(\sigma^\prime).$
The proof of $\Lambda_2 \geq E(\sigma-\sigma^\prime)$ uses the same arguments and so will not be repeated here.
Finally, with \eqref{GeqInq57} in hand, \eqref{ECrev71} follows from the second inequality of \eqref{Limit1m}.
\end{proof}

Analogue of Lemma~\ref{singleDi} is the following.

\begin{lem}
For any minimizing sequence $\{f_n\}_{n\geq 1}$ of $E(\sigma),$ let $M$ be defined by \eqref{limitDEF}.
Then $M$ satisfies $M\not\in (0, \sigma_1+\sigma_2).$
\end{lem}
\begin{proof}
To rule out the dichotomy, suppose to the contrary that
$M\in(0,\sigma_1+\sigma_2).$ Let $\sigma^\prime$ be as in Lemma~\ref{ECrev} and
denote $\sigma^{\prime \prime}=\sigma-\sigma^\prime.$ Lemma~\ref{ECrev} implies
that $E(\sigma)\geq E(\sigma^\prime)+E(\sigma-\sigma^\prime),$ which is same as
\begin{equation}\label{revInEq}
E(\sigma^\prime+\sigma^{\prime \prime})\geq E(\sigma^\prime)+E(\sigma^{\prime \prime}).
\end{equation}
We now deduce a contradiction.
Since $\sigma^\prime, \sigma^{\prime \prime}\in \mathbb{R}_{\geq 0}^2$ with $\sigma^\prime, \sigma^{\prime \prime}\neq \{\mathbf{0}\}$ and
$\sigma^\prime+\sigma^{\prime \prime}\in \mathbb{R}_{> 0}^2,$ we consider the following cases:
(i) $\sigma^\prime, \sigma^{\prime \prime}\in \mathbb{R}_{>0}^2 ;$ (ii) $\sigma^\prime \in \mathbb{R}_{>0}^2$ and $\sigma^{\prime \prime}\in \{0\}\times \mathbb{R}_{>0}$; (iii) $\sigma^\prime \in \mathbb{R}_{>0}\times \{0\} $ and $\sigma^{\prime \prime}\in \mathbb{R}_{>0}^2;$ and
 (iv) $\sigma^\prime \in \mathbb{R}_{>0}\times \{0\}$ and $\sigma^{\prime \prime}\in \{0\}\times \mathbb{R}_{>0}.$
 All other cases coincide with one of these cases after switching the roles of $\sigma^\prime$ and $\sigma^{\prime \prime}.$ We consider each case separately.

\smallskip

\textbf{Case 1.} $\sigma^\prime, \sigma^{\prime \prime}\in \mathbb{R}_{>0}^2.$ Let $\{(z_1^n, z_2^n)\}_{n\geq 1}$ and $\{(w_1^n, w_2^n)\}_{n\geq 1}$
be sequences in $Y^\alpha(\mathbb{R}^N)$ such that
\[
\begin{aligned}
& \|z_1^n\|_{L^2}^2\to {\sigma_1}^\prime,\ \|z_2^n\|_{L^2}^2\to {\sigma_2}^{\prime},\ E(z_1^n, z_2^n)\to E(\sigma^\prime), \\
& \|w_1^n\|_{L^2}^2\to {\sigma_1}^{\prime \prime},\ \|w_2^n\|_{L^2}^2\to {\sigma_2}^{\prime \prime},\ E(w_1^n, w_2^n)\to E(\sigma^{\prime \prime}).
\end{aligned}
\]
To deduce a contradiction in this case, consider the sequences $\{(e_1^n, e_2^n)\}$ and $\{(d_1^n, d_2^n)\}$ in $\mathbb{R}^2$ defined as follows
\[
\begin{aligned}
& e_1^n=\frac{1}{ \|z_1^n\|_{L^2}^2}\left(\mathbb{F}_1(z_1^n)-\mu \iint_{\mathbb{R}^N}D_q(|z_1^n|, |z_2^n|) \ dxdy  \right),\ d_1^n=\frac{1}{\|z_2^n\|_{L^2}^2 } \mathbb{F}_1(z_2^n), \\
& e_2^n=\frac{1}{ \|w_1^n\|_{L^2}^2}\left(\mathbb{F}_2(w_1^n)-\mu \iint_{\mathbb{R}^N}D_q(|w_1^n|||w_2^n|)\ dxdy  \right),\ d_2^n=\frac{1}{\|w_2^n\|_{L^2}^2 } \mathbb{F}_2(w_2^n).
\end{aligned}
\]
Then, we can assume that $(e_1^n, e_2^n)\to (e_1, e_2)$ in $\mathbb{R}^2$
and $(d_1^n, d_2^n)\to (d_1, d_2)$ in $\mathbb{R}^2.$ As in the proof of \eqref{subSING},
three cases may arise: $e_1<e_2, e_1>e_2,$ and $e_1=e_2.$ Suppose that $e_1<e_2.$
Without loss of generality, we may assume that $z_1^n, z_2^n, w_1^n,$ and $w_2^n$ are $\mathbb{R}$-valued, non-negative, and have compact supports. Let $\nu$ be a unit vector in $\mathbb{R}^N$ and define the function $f_2^n$ as follows
\[
\widetilde{z}_2^n(\cdot)=z_2^n(\cdot - b_n \nu),\ f_2^n=\widetilde{z}_2^n+w_2^n,
\]
where $b_n$ is such that $\widetilde{z}_2^n$ and $w_2^n$ have disjoint supports.
Then $\|f_2^n\|_{L^2}^2\to \sigma_2^{\prime }+\sigma_2^{\prime \prime}.$ Define $\{f_1^n\}\subset H^\alpha(\mathbb{R}^N)$ as $f_1^n=T z_1^n,$ where
$T=1+\sigma_1^{\prime \prime}/\sigma_1^{\prime}$ and put $F_n=(f_1^n, f_2^n).$ Then, we have that
\begin{equation}\label{CCFcase}
E(\sigma^\prime+\sigma^{\prime \prime})\leq \lim_{n\to \infty}E(F_n).
\end{equation}
Now, since $\mu>0$ and $q-2>0,$ using the same argument as in the proof of Lemma~\ref{singleDi}, we deduce that
\[
\begin{aligned}
\mathbb{F}_1(f_1^n)-\mu \mathbb{D}_q(F_n) & \leq \mathbb{F}_1(f_1^n)-\mu \iint_{\mathbb{R}^N} D_q(|f_1^n|, |\widetilde{z}_2^n|)\ dxdy\\
& \leq T \left(\mathbb{F}_1(z_1^n)- \mu\iint_{\mathbb{R}^N} D_q(|z_1^n|,|\widetilde{z}_2^n|)\ dxdy\right).
\end{aligned}
\]
Letting $n$ tend to $+\infty$ on both sides of this last inequality, we get
\begin{equation}\label{CCFcase3}
\lim_{n\to \infty}\left( \mathbb{F}_1(f_1^n)-\mu \mathbb{D}_q(F_n) \right) \leq T \sigma_1^{\prime}e_1=\sigma_1^\prime e_1+\sigma_1^{\prime \prime}e_1.
\end{equation}
Using \eqref{CCFcase3} and the assumption $e_1<e_2,$ we see from \eqref{CCFcase} that
\[
\begin{aligned}
E(\sigma^\prime+\sigma^{\prime \prime})& \leq \sigma_1^\prime e_1+\sigma_1^{\prime\prime }e_1+d_1\sigma_2^{\prime}+d_2\sigma_2^{\prime \prime}\\
& = E(\sigma^\prime)+\sigma_1^{\prime\prime }e_1+d_2\sigma_2^{\prime \prime}< E(\sigma^\prime)+E(\sigma^{\prime \prime}),
\end{aligned}
\]
which contradicts \eqref{revInEq}. The proof in the case $e_1>e_2$ uses the same argument and so will not be repeated. Suppose now that $e_1=e_2.$ Two subcases may arise: $d_1\leq d_2$ or $d_1\geq d_2.$ Since both subcases use the same arguments, we
only consider the subcase $d_1\leq d_2.$ Let $f_1^n$ be as in the preceding paragraph and define $f_2^n$ by $f_2^n=S z_2^n,$ where
$S=1+\sigma_2^{\prime \prime}/\sigma_2^{\prime}.$ Then, using the second part of Lemma~\ref{PoLem1}, we can find $\delta>0$ such that
\begin{equation}\label{extraIn}
\mathbb{F}_2(f_2^n)\leq S \mathbb{F}_2(z_2^n)-\delta.
\end{equation}
Using the same argument as in the proof of Lemma~\ref{singleDi}, we also have
\begin{equation}\label{extraIn2}
\mathbb{F}_1(f_1^n)-\mu \mathbb{D}_q(f_1^n, f_2^n) \leq T \left(\mathbb{F}_1(z_1^n)-\mu \iint_{\mathbb{R}^N} D_q(|z_1^n|, |z_2^n|)\ dxdy \right).
\end{equation}
Now, since $\|f_1^n\|_{L^2}^n\to \sigma_1^\prime + \sigma_1^{\prime \prime}$ and $\|f_2^n\|_{L^2}^n\to \sigma_2^\prime + \sigma_2^{\prime \prime},$
as a consequence of inequalities \eqref{extraIn} and \eqref{extraIn2}, one obtains that
\[
\begin{aligned}
E(\sigma^\prime+\sigma^{\prime \prime})& \leq \lim_{n\to \infty}E(f_1^n, f_2^n)\leq T e_1\sigma_1^\prime+S \sigma_2^\prime d_1-\delta \\
& = e_1\sigma_1^\prime + \sigma_2^\prime d_1+\frac{\sigma_1^{\prime \prime} } { \sigma_1^{\prime}}e_1\sigma_1^\prime+\frac{\sigma_2^{\prime \prime} } { \sigma_2^{\prime}}\sigma_2^\prime d_1-\delta
 = E(\sigma^\prime)+\sigma_1^{\prime \prime}e_1+\sigma_2^{\prime \prime}d_1-\delta.
\end{aligned}
\]
Using $e_1=e_2$ and $d_1\leq d_2,$ this last inequality implies that
\[
E(\sigma^\prime+\sigma^{\prime \prime})\leq E(\sigma^\prime)+E(\sigma^{\prime \prime})-\delta,
\]
which gives, $E(\sigma^\prime+\sigma^{\prime \prime})< E(\sigma^\prime)+ E(\sigma^{\prime \prime}),$ which again
contradicts \eqref{revInEq}.

\smallskip

\textbf{Case 2.} $\sigma^\prime\in \mathbb{R}_{>0}^2$
and $\sigma^{\prime \prime}\in \{0\}\times \mathbb{R}_{>0}.$
Let $\{(z_1^n, z_2^n)\}$ and $\{(w_1^n, w_2^n)\}$ be sequences in $Y^\alpha(\mathbb{R}^N)$ such that
\[
\begin{aligned}
& \|z_1^n\|_{L^2}^2\to \sigma_1^\prime,\ \|z_2^n\|_{L^2}^2\to \sigma_2^\prime,\ E(z_1^n, z_2^n)\to E(\sigma^\prime),\\
& \|w_1^n\|_{L^2}^2\to 0,\ \|w_2^n\|_{L^2}^2\to \sigma_2^{\prime \prime},\ E(w_1^n, w_2^n)\to E(\sigma^{\prime \prime}).
\end{aligned}
\]
Define $\{(K_1^n, K_2^n)\}_{n\geq 1}\subset \mathbb{R}^2$ as follows
\[
\begin{aligned}
& K_1^n=\frac{1}{\|w_2^n\|_{L^2}^2}\left( \mathbb{F}_2(w_2^n)-\mu \iint_{\mathbb{R}^N} D_q(|z_1^n|, |w_2^n|)\ dxdy \right),\\
& K_2^n=\frac{1}{\|z_2^n\|_{L^2}^2}\left( \mathbb{F}_2(z_2^n)-\mu \iint_{\mathbb{R}^N} D_q(|z_1^n|, |z_2^n|)\ dxdy \right).
\end{aligned}
\]
Then, as before, we can assume that $(K_1^n, K_2^n)\to (K_1, K_2)$ in $\mathbb{R}^2.$ Three
cases may arise: $K_1<K_2, K_1>K_2,$ and $K_1=K_2.$ Suppose that $K_1<K_2.$ Let $f_2^n=B w_2^n,$ where $B=1+\sigma_2^\prime/\sigma_2^{\prime \prime}.$ Then
\begin{equation}\label{case3CC}
E(\sigma^\prime + \sigma^{\prime \prime})\leq \lim_{n\to \infty}E(z_1^n, f_2^n).
\end{equation}
Using the same argument as in the proof of Lemma~\ref{singleDi}, we can show that
\[
\mathbb{F}_2(f_2^n)-\mu \mathbb{D}_q(z_1^n, f_2^n)\leq B \left(\mathbb{F}_2(w_2^n)-\mu \iint_{\mathbb{R}^N} D_q(|z_1^n|, |w_2^n|)\ dxdy \right).
\]
Using this and the assumption $K_1<K_2$ into \eqref{case3CC}, one obtains that
\[
\begin{aligned}
E(\sigma^\prime + \sigma^{\prime \prime})& \leq \lim_{n\to \infty} \mathbb{F}_1(z_1^n)+B \lim_{n\to \infty} \left(\mathbb{F}_2(w_2^n)-\mu \iint_{\mathbb{R}^N}D_q(|z_1^n|, |w_2^n|)\ dxdy \right)\\
&= E(\sigma^{\prime \prime})+\lim_{n\to \infty} \mathbb{F}_1(z_1^n)+\sigma_1^\prime K_1< E(\sigma^{\prime \prime})+E(\sigma^{\prime}),
\end{aligned}
\]
which is a contradiction. The case $K_1>K_2$ also leads to a contradiction with the same argument. Finally, consider the case $K_1=K_2.$
Let $f_2^n=B w_2^n$ as above. Then we have that
\begin{equation}\label{Case3cC}
E(\sigma^\prime + \sigma^{\prime \prime})\leq \lim_{n\to \infty}E(z_1^n, f_2^n).
\end{equation}
Making use of Lemma~\ref{PoLem1} and using the same argument that we used in the case $e_1=e_2$ before, it follows that
\[
E(z_1^n, f_1^n)\leq \mathbb{F}_1(z_1^n)+B \mathbb{F}(w_2^n)-\mu B \iint_{\mathbb{R}^N} D_q(|z_1^n|, |w_2^n|)\ dxdy-\delta.
\]
Using this and $K_1=K_2$ into \eqref{Case3cC}, we obtain that
\[
\begin{aligned}
E(\sigma^\prime + \sigma^{\prime \prime}) \leq E(\sigma^{\prime \prime})+\lim_{n\to \infty}\mathbb{F}_1(z_1^n)+\frac{\sigma_2^\prime }{\sigma_2^{\prime \prime}}K_1 \sigma_2^{\prime \prime}-\delta
 =E(\sigma^{\prime \prime})+E(\sigma^{\prime }) -\delta,
\end{aligned}
\]
which gives $E(\sigma^\prime + \sigma^{\prime \prime})< E(\sigma^{\prime \prime})+E(\sigma^{\prime }),$ this contradicts \eqref{revInEq}.

\smallskip

\textbf{Case 3.} $\sigma^\prime \in \mathbb{R}_{>0}\times \{0\}$ and $\sigma^{\prime \prime}\in \mathbb{R}_{>0}^2.$ The proof in this case uses the same argument as in the Case 2 and so will not not be repeated here.

\smallskip

\textbf{Case 4.} $\sigma^\prime\in \mathbb{R}_{>0}\times \{0\}$ and $\sigma^{\prime \prime}\in \{0\}\times \mathbb{R}_{>0}.$ In this case, we have that
$E(\sigma^\prime)=J_1(\sigma_1^\prime)$ and $E(\sigma^{\prime \prime})=J_2(\sigma_2^{\prime \prime}),$ where
\[
J_1(\sigma_1^\prime)=\inf_{f\in \Sigma_{\sigma_1^\prime} }J(f)|_{(a,\lambda, p)=(0, \lambda_1, p_1)}\ \ \mathrm{and}\ \ J_2(\sigma_2^{\prime \prime})=\inf_{f\in \Sigma_{\sigma_2^{\prime \prime}} }J(f)|_{(a, \lambda, p)=(0, \lambda_2, p_2)}.
\]
For $\sigma_1^\prime>0$ and $\sigma_2^{\prime \prime}>0,$ let $u^\prime$ and $u^{\prime \prime}$ be such that
$J_1(\sigma_1^\prime)=J(u^\prime)|_{(a,\lambda, p)=(0, \lambda_1, p_1)}$ and $J_2(\sigma_2^{\prime \prime})=J(u^{\prime \prime})|_{(a,\lambda, p)=(0, \lambda_2, p_2)}.$ Then, it is obvious that $\mathbb{D}_q(u^\prime, u^{\prime \prime})>0$ and
\[
E(\sigma_1^{\prime}, \sigma_2^{\prime \prime})<J(u^\prime)|_{(a,\lambda, p)=(0, \lambda_1, p_1)}+J(u^{\prime \prime})|_{(a,\lambda, p)=(0, \lambda_2, p_2)}=J_1(\sigma_1^\prime)+J_2(\sigma_2^{\prime \prime}),
\]
which again contradicts \eqref{revInEq}.
\end{proof}
Analogue of Lemma~\ref{shiftLem} hold in
the present context without change of statement and an obvious
modification in the proof. Thus, all the preliminaries for the proof of Theorem~\ref{CHMain} and its corollary
have been established.
The proofs of Theorem~\ref{CHMain}, except Statements 3 and 4, and its Corollary are
now standard and
so will not be repeated here.

\smallskip

To prove Statement 3 of Lemma~\ref{CHMain}, let $(u_1,u_2)\in M_\sigma$ for any $\sigma\in \mathbb{R}_{>0}^2.$ Then there exists a pair $\omega=(\omega_1, \omega_2)\in \mathbb{R}^2$ such that $(\omega, u_1, u_2)$ solves
the Euler-Lagrange differential equations
\begin{equation}\label{CfODEp4}
\left\{
\begin{aligned}
  & (-\Delta)^\alpha u_1+\omega_1u_1=\lambda_1\left( \mathcal{K}_\beta\star |u_1|^{p_1}\right)|u_1|^{p_1-2}u_1 +c \left(\mathcal{K}_\beta\star |u_2|^{q}\right) |u_1|^{q-2}u_1,\\
   & (-\Delta)^\alpha u_2+\omega_2u_2=\lambda_2\left( \mathcal{K}_\beta\star |u_2|^{p_2}\right)|u_2|^{p_2-2}u_2 +c \left(\mathcal{K}_\beta\star |u_1|^{q}\right) |u_2|^{q-2}u_2,
\end{aligned}
\right.
\end{equation}
where $\mathcal{K}_\beta(x)=|x|^{\beta-N}$ for $x\in \mathbb{R}^N.$ The first equation of \eqref{CfODEp4} gives
\begin{equation}\label{CfODEp7}
\|(-\Delta)^{\alpha/2}u_1\|_{L^2}^2-\lambda_1\mathbb{D}_{p_1}(u_1)-c \mathbb{D}_{q}(u_1, u_2) = -\omega_1\|u_1\|_{L^2}^2
\end{equation}
Applying Lemma~\ref{LMnegL} for $(f_1^n, f_2^n)=(u_1,u_2),$ we have that
\[
\|(-\Delta)^{\alpha/2}u_1\|_{L^2}^2-\mu_1\mathbb{D}_{p_1}(u_1)-\mu \mathbb{D}_{q}(u_1, u_2)<0.
\]
Since $\lambda_1=2\mu_1 p_1>\mu_2$ and $c=\mu q>\mu,$ it follows from \eqref{CfODEp7} that $-\omega_1\sigma_1<0.$ This implies that $\omega_1>0.$ Similarly, we have that $\omega_2>0.$

\smallskip

The proofs that $(|u_1|, |u_2|)\in M_\sigma$ and $(|u_1|^\ast, |u_2|^\ast)\in M_\sigma$
whenever $u\in M_\sigma$ follow from the facts \eqref{symPO2}$-$\eqref{symPO2}.
To show that $|u_1|>0$ and $|u_2|>0$ on $\mathbb{R}^N,$ denote
$\widetilde{u}_1=|u_1|$ and $\widetilde{u}_2=|u_2|.$ Then, the function $(\widetilde{u}_1,\widetilde{u}_2)$ satisfies
the system of the form
\begin{equation}\label{CfODEp19}
\left\{
\begin{aligned}
  & (-\Delta)^\alpha \widetilde{u}_1+\omega_1\widetilde{u}_1=f_1(\widetilde{u}_1, \widetilde{u}_2)\ \ \mathrm{in}\ \mathbb{R}^N,\\
   & (-\Delta)^\alpha \widetilde{u}_2+\omega_2\widetilde{u}_2=f_2(\widetilde{u}_1, \widetilde{u}_2)\ \ \mathrm{in}\ \mathbb{R}^N,
\end{aligned}
\right.
\end{equation}
where $(\omega_1, \omega_2)$ is the same pair of numbers as in \eqref{CfODEp4}.
Since $\omega_1>0$ and $\omega_2>0,$ we have the convolution representation
\[
\left\{
\begin{aligned}
& \widetilde{u}_1=\mathcal{W}_{\omega_1}\star f_1(\widetilde{u}_1, \widetilde{u}_2)=\int_{\mathbb{R}^N}\mathcal{W}_{\omega_1}(x-x^\prime)f_1(\widetilde{u}_1, \widetilde{u}_2)(x^\prime)\ dx^\prime,\\
&\widetilde{u}_2=\mathcal{W}_{\omega_2}\star f_2(\widetilde{u}_1, \widetilde{u}_2)=\int_{\mathbb{R}^N}\mathcal{W}_{\omega_2}(x-x^\prime)f_2(\widetilde{u}_1, \widetilde{u}_2)(x^\prime)\ dx^\prime,
\end{aligned}
\right.
\]
where $\mathcal{W}_\tau(x)$ is as defined in \eqref{lagra11}.
Since the functions $f_1, f_2$ are everywhere nonnegative and not identically zero, it
follows that $\widetilde{u}_1>0$ and $\widetilde{u}_2>0$ on $\mathbb{R}^N.$ This concludes the proof of Statement 4.

\end{document}